%% file: main.tex
\documentclass[11pt]{article}
\usepackage[margin=1.1 in]{geometry}

\input{ex_shared}

\usepackage{empheq}
\usepackage[theorems,skins]{tcolorbox}

\makeatletter
\newcommand{\myitem}[1]{%
\item[#1]\protected@edef\@currentlabel{#1}%
}
\makeatother

\newenvironment{customproof}[1]
  {\par\noindent\textit{Proof of #1.}\space}
  {\hfill$\square$}

\begin{document}
	\date{}
	\maketitle
    	
	\begin{abstract}
	In this work, we consider constrained stochastic optimization problems under \textit{hidden convexity}, i.e., those that admit a convex reformulation via non-linear (but invertible) map $c(\cdot)$. A number of non-convex problems ranging from optimal control, revenue and inventory management, to convex reinforcement learning all admit such a hidden convex structure. Unfortunately, in the majority of applications considered, the map $c(\cdot)$ is unavailable or implicit; therefore, directly solving the convex reformulation is not possible. On the other hand, the stochastic gradients with respect to the original variable are often easy to obtain. Motivated by these observations, we examine the basic projected stochastic (sub-)\,gradient methods for solving such problems under hidden convexity. We provide the first sample complexity guarantees for global convergence in smooth and non-smooth settings. Additionally, in the smooth setting, we improve our results to the last iterate convergence in terms of function value gap using the momentum variant of projected stochastic gradient descent.  
	\end{abstract}
	
	

\tableofcontents

	\section{Introduction}

	We study constrained  stochastic  optimization  
	\begin{equation}
		\label{problem:original}
		\min_{x\in\cX}~F(x) := \Expu{\xi \sim \cD}{f(x, \xi)}, 
	\end{equation}
	where $\cX $ is a closed convex subset of $\R^d$, $\xi$ is a random variable following an unknown distribution $\cD$, and $F(\cdot)$ is possibly non-convex in $x$. Our central structural assumption about \eqref{problem:original} is that it admits a \textit{convex reformulation} of the form
	\begin{equation}
		\label{problem:reformulation}
		\min_{u\in \cU}~H(u):= F(c^{-1}(u)), 
	\end{equation}
	where $H(\cdot)$ is a convex function defined on a closed convex set $\mathcal{U} \subset \R^d$, and $c : \mathcal{X} \rightarrow  \mathcal{U}$ is an invertible map (with its inverse denoted by $c^{-1}(\cdot)$). This property is often referred to as \emph{hidden convexity} and frequently appears in various modern applications, for example, policy optimization in convex reinforcement learning and optimal control~\cite{zhang2020variational,sun2021learning,ying2023policy}, generative models~\cite{kobyzev2020norm_flows}, supply chain and revenue management~\cite{feng2018supply,chen2022efficient}, training neural networks~\cite{wang2020hidden,ergen2021global}. In the optimization literature, hidden convexity has been identified much earlier and dates back to at least 1990s in the context of quadratic optimization~\cite{stern1995indefinite,ben1996hidden}. Since then, several works have developed various tools to identify such property~\cite{ben2011hidden,ben2014hidden}. More recently, hidden convexity has been established for a wider classes of non-convex programs~\cite{Tibshirani2016hidden,xia2020survey,chancelier2021conditional} and for non-monotone games~\cite{emmanouil2021solving,mladenovic2021generalized}, to name just a few. {\color{black} The effect of general smooth parameterizations on the landscape of non-convex problems was recently studied in \cite{levin2022effect}.}

	Despite the existence of the convex reformulation, the transformation function $c(\cdot)$ is usually hard to compute or even unknown, and one cannot readily solve the convex reformulation. This motivates the use of (sub-)\,gradient methods that optimize $F(\cdot)$ directly over the variable $x \in \cX$. Perhaps, the most basic algorithm is the projected stochastic (sub-)\,gradient method (SM). Starting from $x^0 \in \cX$, SM generates a sequence $x^t$ via 
\begin{equation}\label{eq:subgradient_method}
	x^{t+1} = \Pi_{\mathcal{X}}(x^t-\stepsize  g(x^{t}, \xi^{t}) ) ,  
\end{equation}
where $g(x^{t}, \xi^{t})$ denotes an unbiased estimate of the (sub-)\,gradient of $F(\cdot)$ at a point $x^t$, and $\Pi_{\mathcal{X}}(\cdot)$ is the Euclidean projection onto a convex set $\cX$, see \Cref{sec:notations} for details. When $F(\cdot)$ is differentiable, this reduces to Projected SGD. Stochastic (sub-)\,gradient methods and their numerous variants have a long history of development since the first works on stochastic approximation appeared in 1950s \cite{robbins1951stochastic,Kiefer_Wolfowitz_1952,blum1954multidimensional,chung1954stochastic}.
Their analyses are richly documented for addressing convex problems and general nonconvex problems (refer to Appendix C for detailed summary); however, their convergence behaviors when dealing with hidden convexity still elude precise understanding.

    Although hidden convexity has been previously identified in certain applications, the analysis of gradient methods under this condition is mostly done on a case by case basis for specific applications and often requires strong additional assumptions \cite{zhang2021convergence,barakat2023RLwGU,chen2022network,chen2022efficient}. In this work, we formally consider general-purpose stochastic optimization under hidden convexity and study the sample complexities for solving such problems through projected stochastic (sub-)\,gradient method and the like.    
    
    \textbf{Contributions:}
	\begin{enumerate}
		\item We identify  key properties of hidden convex optimization and demonstrate how these conditions can be used to derive global convergence of gradient methods. 
		\item In the general non-differentiable case, we analyze convergence of the projected stochastic sub-gradient method (SM) and obtain $\wt \cO\rb{ \frac{ \ell }{\mu_c^2 } \frac{1}{\varepsilon}   + \frac{ \ell G_F^2 }{\mu_c^4  } \frac{1}{\varepsilon^3}  }  $ sample complexity for driving the Moreau envelope of \eqref{problem:original} $\varepsilon$-close to the optimal value in expectation. Here $\ell$ is weak convexity parameter of $F(\cdot)$, $G_F$ is the bound on (stochastic) subgradients, and $\mu_c$ is the modulus of hidden convexity of $F(\cdot)$ that relates to the invertibility of mapping $c(\cdot)$; see~\Cref{sec:HC_class} and \Cref{sec:subgradient_method} for details. To our knowledge, it is the first result to address the non-differentiable setting under hidden convexity.
        \item Next, we specialize our results to the differentiable smooth setting, and obtain a similar sample complexity for Projected SGD, replacing $G_F$ by the variance of stochastic gradients $\sigma$. Furthermore, we analyze the momentum variant of Projected SGD improving our result to the last iterate convergence in terms of the function value gap, i.e., we have $ \Exp{F(x^T) - F^*} \leq \varepsilon$ after $T = \wt \cO\rb{ \frac{ L }{\mu_c^2 } \frac{1}{\varepsilon}   + \frac{ L \sigma^2 }{\mu_c^4  } \frac{1}{\varepsilon^3}  }  $ iterations/stochastic gradient calls, where $L$ is the Lipschitz constant of $\nabla F(x)$.  
        \item In the presence of strong convexity of the reformulated problem, we further improve the sample complexity for all above mentioned algorithms. For instance, Projected SGD attains $\wt \cO\rb{ \frac{ L }{\mu_c^2 \mu_H }   + \frac{ L \sigma^2 }{\mu_c^4 \mu_H^2  } \frac{1}{\varepsilon}  }  $ sample complexity for achieving an $\varepsilon$-optimal solution, where $\mu_H>0$ corresponds to the strong convexity of $H(\cdot)$.

	\end{enumerate}	

Importantly, we show that all studied algorithms provably converge in online fashion, using only one stochastic gradient at every iteration.

    \subsection{Related work} 
    
    {\color{black} Perhaps, one of the first works using hidden convexity to analyze iterative methods is \cite{nesterov2006cubic}, which established convergence rates of cubic regularized Newton method. However, it remains unclear how to extend their results to non-smooth or stochastic setting.} The most closely related to our work are \cite{zhang2020variational,zhang2021convergence,barakat2023RLwGU,chen2022network,ghai2022_alg_equivalence}, which analyze gradient methods under similar structural assumptions in the context of specific applications. 
    
    \textit{Policy gradient methods in RL.} Several work \cite{zhang2020variational,zhang2021convergence,barakat2023RLwGU} exploited properties similar to hidden convexity in reinforcement learning (RL) applications and analyzed policy gradient (PG) type methods with global convergence guarantees. In \cite{zhang2020variational}, the authors considered a PG method with projection, but it is only limited to the case where the exact gradients are available. It is unclear how to extend the technique in their work to the case of stochastic gradients with bounded variance (without resorting to large batches). Next, \cite{zhang2021convergence} considered the stochastic setting and proposed a variance-reduced PG method with truncation using large batches of trajectories. Recently, \cite{barakat2023RLwGU} removed the requirement for large batches using a normalized variance-reduced PG method. However, their results are difficult to extend to the constrained case due to the normalization. Moreover, both works \cite{zhang2021convergence,barakat2023RLwGU} utilize variance reduced estimators, which require additional smoothness assumptions for theoretical analysis.

    \textit{Stochastic gradient methods in revenue management.} 
    A different line of works \cite{chen2022efficient,chen2022network} considered hidden convex objectives in revenue management and studied global convergence of gradient-based methods over $\mathcal{X}$. For a special revenue management problem, \cite{chen2022efficient} introduced a preconditioned gradient-based method that obtains an $\wt{\mathcal{O}}(\epsilon^{-2})$ sample complexity under the assumptions that the domain $\mathcal{X}$ is a box constraint, the transformation function $c(x) = \Exp{c(x, \xi)}$ is separable and the additional access to $c(x, \xi)$ is available. Leveraging the box constraint structure, \cite{chen2022efficient} also analyzed Projected SGD and derive $\wt{\mathcal{O}}(\epsilon^{-4})$ sample complexity. 
    In contrast, we show that Projected SGD can achieve a better $\wt \cO(\epsilon^{-3})$ sample complexity for a general convex compact constraint $\cX$, and further extend the results to non-smooth setting.

    \textit{Nonconvex online learning.} Recently, \cite{ghai2022_alg_equivalence} considered a structural property similar to hidden convexity and imposed strong assumptions on the reparameterization map $c(\cdot)$ (see Assumptions 1, 2 and 4 therein) under which non-convex online gradient descent in the original space $\cX$ is equivalent to online mirror descent for the (convex) reformulated problem. Such equivalence allows them to demonstrate an $\cO(T^{\nfr{2}{3}})$ regret bound. Instead, we directly derive the last iterate convergence in the function value using a different technique and make less restrictive assumptions on $c(\cdot)$, which allows us to cover a wide range of applications.

 \textit{Related structural assumptions.} 
        We mention that several other non-convex structural assumptions have been explored in optimization literature that also ensure global convergence, including essential strong convexity \cite{liu2014asynchronous}, quasar (strong) convexity \cite{hinder2020near}, restricted secant inequality \cite{zhang2013gradient}, error bounds \cite{luo1993error}, quadratic growth \cite{bonnans1995second}, Polyak-{\L}ojasiewicz (P{\L}) condition \cite{polyak1963gradient,lojasiewicz1963property} -- also known as global Kurdyka-{\L}ojasiweicz (K{\L}) or gradient domination condition. For an in-depth discussion on the relationships between these properties, refer to ~\cite{karimi2016linear,rebjock2023fast} and references therein. Notably, the P{\L} condition along with its various generalizations to constrained minimization such as Proximal-P{\L} \cite{karimi2016linear} and variational gradient dominance \cite{xiao_PGM} have gained popularity in the recent years. The convergence of gradient-based methods under the P{\L}-type conditions has been extensively analyzed, e.g., in the deterministic setting~\cite{karimi2016linear,Yue_PL_LB_2023} and the stochastic setting~\cite{KL_PAGER_Fatkhullin,fontaine2021convergence,khaled2020better,scaman22a,Li_RR_under_KL_2023,chouzenoux2023kurdyka,yuan2022general,Fatkhullin_SPGM_FND_2023}. Unfortunately, despite a few examples~\cite{fazel2018global,Optimizing_LQR_21,ding2022global,xiao_PGM} that show some variants of the P{\L} condition hold,  how to verify P{\L}-like conditions for non-convex problems remains a big question in general. The situation becomes even more challenging, when dealing with constrained optimization and/or non-differentiable objectives, where a suitable generalization of the gradient dominance needs to be introduced and carefully studied. 
        {\color{black} 
        For a non-differentiable $F$, convergence results under the  gradient domination condition often make very strong algorithmic dependent assumptions (e.g., strong descent property), and are difficult to extend to stochastic setting \cite{absil2005convergence,attouch2009convergence,attouch2010proximal,frankel2015splitting}. In the constrained setting, a suitable generalization of the gradient dominance is often algorithmic dependent. Specifically, one needs to replace the gradient norm by a suitable generalized notion of stationarity, e.g. Frank-Wolfe gap or gradient mapping \cite{xiao_PGM}. However, verifying such algorithmic dependent assumptions can be difficult in specific applications.} Unlike the gradient domination condition, the hidden convexity considered in this work is a very natural property, which easily extends to constrained optimization and non-differentiable objectives.

	 \subsection{Notations and organization}\label{sec:notations}
In the following, we briefly revisit some basic notations from the convex analysis. 
	Throughout, we denote by $\ve{\cdot}{\cdot}$ the inner product in $\R^d$ along with its induced Euclidean norm $\norm{\cdot}$. For a real valued matrix $A\in \R^{m\times n}$, we denote by $\opnorm{\cdot}$ its operator norm, i.e.,  $\opnorm{A} := \max_{\norm{x}\leq 1}\norm{A x}$. {\color{black} We denote the interior and the boundary of $\cX$ as $\operatorname{int} \cX$ and $\partial \cX$ respectively.} 
	The map $c: \cX \rightarrow \cU$ is called invertible if there exists a map $c^{-1}: \cU \rightarrow \cX$ (called inverse) such that $c^{-1}(c(x)) = x$ for any $x \in \cX$ and $c( c^{-1}(u) ) = u$ for any $u \in \cU$. For any $u, v\in\cU$ and any $\lambda\in[0,1]$, if  $(1-\lambda) u + \lambda v \in \cU$, we say $\mathcal{U}$ is convex. We denote the diameter of $\cU$ as $D_{\cU}:= \sup_{u, v\in \cU}\pnorm{u - v}$.
	For a function $H: \cU \rightarrow \R \cup \{+\infty\} $, if there exists $\mu_H \geq 0$ such that 
	for all $u, v \in \cU$ and $ \lambda \in [0,1]$, it holds 
	$	H( (1-\lambda)u + \lambda v ) \leq  (1- \lambda) H(u) +  \lambda  H(v)  -   \fr{(1-\lambda)\lambda \mu_H}{2} \sqpnorm{ u - v }, $
	we call $H$ convex on  $\cU$ if $\mu_H = 0$, and $\mu_H$-strongly convex on $\cU$ if $\mu_H>0$.

A function $F: \cX \rightarrow\R \cup \{+\infty\}$ is $\ell$-weakly convex ($\ell$-WC) if for any fixed $y\in \cX$, $F_{\ell}(x,y) := F(x) + \frac{\ell}{2} \, \sqnorm{x- y}$ is convex in $x \in \cX$.
If $x\notin \cX$, we assign $F(x) = +\infty$. The (Fr\'echet) sub-differential of $F$ at $x\in \cX$ is $\partial F(x) := \left\{g \in \mathbb{R}^d \mid  F(y) \geq F(x)+\langle g, y-x\rangle+o(\|y-x\|), \forall y \in \R^d  \right\}.$ The elements $g \in \partial F(x)$ are called sub-gradients of $F$ at $x$, see \cite{davis2019proximally} for alternative equivalent definitions of the sub-differential set for $\ell$-WC functions. A differentiable function $F: \cX \rightarrow \R$ is $L$-smooth on $\cX \subset \R^d$ if its gradient is $L$-Lipschitz continuous on the set $\cX$, i.e., it holds $\norm{\nabla F(x) - \nabla F(y)} \leq L \norm{x - y}$ for all $x, y\in \cX$. For a convex set $\cX \subset \R^d$, the projection of a point $y\in \R^d$ onto $\cX$ is  $\Pi_{\cX}(y) := \argmin_{x\in\cX}\norm{y - x}$. We denote $\delta_{\cX}$ as the indicator function of a set $\cX \subset \R^d$ and define $\delta_{\cX}(x)  = 0$ if $x\in \cX$ and $\delta_{\cX}(x)  = +\infty$ otherwise. 
    We define by $\cX^* \subset \cX$ the set of optimal points of $\min_{x\in\cX} F(x)$ and by $F^*$ its optimal value. A point $\bar{x} \in \cX$ is called a stationary point of a weakly convex function $F: \cX \rightarrow \R$ if $0 \in \partial (F +  \delta_{\cX})(\bar{x})$.
	For any function $\Phi$ and a real $\rho > 0$, we define the Moreau envelope and the proximal mapping as follows
	\begin{eqnarray}
		\Phi_{1/\rho}(x) := \min_{y\in \R^d} \left\{\Phi(y) + \frac{\rho}{2} \sqnorm{ y - x }  \right\},  \quad \operatorname{prox}_{  \Phi/\rho }(x) := \argmin_{y\in \R^d } \left\{  \Phi(y) + \frac{\rho}{2} \sqnorm{ y - x } \right\} . \notag 
	\end{eqnarray}
		

  The rest of the paper is organized as follows. We formally introduce the hidden convex function class in \Cref{sec:HC_class} followed by motivating examples. The properties of hidden convex optimization, which are useful to analyze global convergence of gradient methods are in \Cref{sec:prop_HC}. Our main global convergence results are in \Cref{sec:subgradient_method,sec:PSGD,sec:PSGDM} for subgradient methods, Projected SGD, and Projected SGD with momentum, respectively. The conclusion follows in \Cref{sec:conclusions}.

	\section{Hidden Convex Problem Class}\label{sec:HC_class}
 
The existence of a convex reformulation \eqref{problem:reformulation} for the problem \eqref{problem:original} signifies its representation as a compositional optimization in the form: 
	\begin{equation}
		\label{problem:original_composition}
		\min_{x\in\cX}~F(x) := H(c(x)) . 
	\end{equation}
Formally, we make the following definition. 
 \begin{definition}
The  above problem is called {hidden convex} with modulus $\mu_c>0, \mu_H\geq 0$ (or function $F$ is hidden convex on $\cX$) if its components satisfy the following underlying conditions. 
	\begin{enumerate}
		\myitem{C.1}\label{C1} The domain $\mathcal{U}=c(\mathcal{X})$ is convex, the function $H: \cU \rightarrow \R$ satisfies that for all $u, v \in \cU$ and $ \lambda \in [0,1]$, 
  \begin{equation}\label{assumption:H_conv}
  	H( (1-\lambda)u + \lambda v ) \leq  (1- \lambda) H(u) +  \lambda  H(v)  -   \fr{(1-\lambda)\lambda \mu_H}{2} \sqpnorm{ u - v }, 
  \end{equation}
  and \eqref{problem:reformulation} admits a solution $u^* \in \cU$.  
  
		\myitem{C.2} The map $c: \cX \rightarrow \cU$ is invertible. There exists $\mu_c>0$ such that  for all $x, y \in \cX$ it holds  \label{C2}
		\begin{equation}\label{assumption:g_inv}
			\pnorm{c(x) - c(y) } \geq \mu_c \pnorm{x - y }.
		\end{equation}
	\end{enumerate}
In particular, if $\mu_H>0$, we say the above problem is {$(\mu_c,\mu_H)$-hidden strongly convex}.  
\end{definition}

 {\color{black}
\textit{Remarks on condition C.2.} First, we notice that \eqref{assumption:g_inv} is equivalent to $1/\mu_c$-Lipschitz continuity of the inverse map $c^{-1}$, i.e.,
\begin{equation}
\label{eq:lipschitz_c_inverse}
    \pnorm{c^{-1}(u) - c^{-1}(v) } \leq \mu_c^{-1} \pnorm{u - v } \qquad \text{for all } u, v \in \cU. 
\end{equation}
Second, if the inverse map $c^{-1}(\cdot)$ is continuously differentiable\footnote{This can be verified using a global inverse function theorem \cite{krantz2002implicit,ivanov2023hadamard} if $c(\cdot)$ is continuously differentiable.} on $\cU$, then there are simple necessary and sufficient conditions, which can help to verify \eqref{assumption:g_inv}. 
Using standard arguments (see, e.g., Section 1.1.2. in \cite{polyak1987introduction} or Section 1.2.2. in \cite{nesterov2018lectures}), we can show that 
\eqref{eq:lipschitz_c_inverse} is equivalent to
\begin{equation}\label{eq:C2_sufficient}
 \norm{\nabla c^{-1}(u)}_{\text{op}} \leq 1/\mu_c \qquad \text{for any } u \in \cU .  
\end{equation}


 When $c(\cdot)$ is continuously differentiable on $\cX$, condition \eqref{assumption:g_inv} implies $\norm{\nabla c(x) z } \geq \mu_c \norm{z} $ for any $x \in \cX$ and $z \in \R^d$, see the proof of \Cref{le:HC_KL} for details.
 We remark that, in general, the map in \cref{C2} is not required to be continuously differentiable and the results of \Cref{sec:subgradient_method} do not require a differentiable $c(\cdot)$. 
 
}
 \textit{Simple examples.}  Notice that the hidden convex problem class includes the convex problem as a special case  when the transformation map $c(\cdot)$ is identical. In addition, it also includes many non-convex problems. For a simple example, let $0<\delta\leq 1$ and consider $\cX = [\delta, 1]$, $c(x) = x^2$, $H(u) = - u$. Then $F(x) = -x^2$, albeit concave,  is hidden convex on $\cX$ by the construction. Another simple example considers $0<\delta < \pi$ and $\cX = [\delta, 2\pi - \delta]$,  {\color{black}$c(x) = \cos(x/2)$, $H(u) = 2 u^2 - 1$}. The obtained composition $F(x) = 2 \cos^2(x/2) - 1 = \cos(x)$ is hidden strongly convex on $\cX$ with $\mu_c = \sin(\delta/2) / 2$, although it is both non-convex and non-concave on $\cX$. {\color{black} From the above toy example, it already becomes clear that the condition $\mu_c > 0$ in \cref{C2} is necessary for global convergence of Projected GD (e.g., merely assuming an invertible $c(\cdot)$ is not sufficient). Indeed, by taking $\delta = 0$, we have $\mu_c = 0$, but this generates a spurious stationary point at $x = 0$, which is also a maxima, where GD gets stuck.} 
 
 In what follows, we present several more practical problems, which belong to our hidden convex  class.

	\subsection{Non-linear least squares \cite{nocedal1999numerical,nesterov2006cubic,drusvyatskiy2019efficiency}}
	Consider solving a system of nonlinear equations, e.g., $c(x)=0$ with $c(x) = (c_1(x), \ldots, c_d(x))^{\top}$ for $x\in\R^d$. This problem can be equivalently formulated as 
	\begin{equation}\label{eq:nonlinear_LS}
     \min_{x\in \R^d} \sum_{i=1}^{d} c_i^2(x)  \qquad \text{or }  \qquad  \min_{x\in \R^d}  \max_{1\leq i \leq d } | c_i(x) |.
	\end{equation}
 When $c(\cdot)$ is invertible and condition \eqref{assumption:g_inv} holds, it belongs to the hidden convex optimization class. {\color{black} Despite having a benign reformulation, such problems can be extremely difficult to solve to global optimum. For example, Jarre~\cite{jarre2013nesterov} study the following function
 \begin{equation}\label{eq:Nesterov_Rosenbrock}
 F(x) = \frac{1}{4}(x_1 - 1)^2 + c \sum_{i=1}^{d-1} (x_{i+1} - 2 x_i^2 - 1)^2 
\end{equation}
 for some constant $c > 0$ (independent of dimension). Despite being hidden convex and having a unique minima, it is shown that any descent method initialized at $x^0 = (-1, 1, \ldots, 1)$ takes an exponential number of iterations (in dimension $d$) to reduce the function value by $75 \%$. 
 
 In practice, solving problems of type \eqref{eq:nonlinear_LS} can be challenging even when the dimension is moderate and other challenges such as non-smoothness of $c(x)$ and $H(u)$ persist, see e.g., non-smooth variants of \eqref{eq:Nesterov_Rosenbrock} in \cite{gurbuzbalaban2012nesterov}. In \cref{fig:chebyshev_plot}, we illustrate the contour plot of such functions and their convex reformulations for dimension $d = 2$. 
}

 \begin{figure}
 	\centering
 	\includegraphics[width=\textwidth]{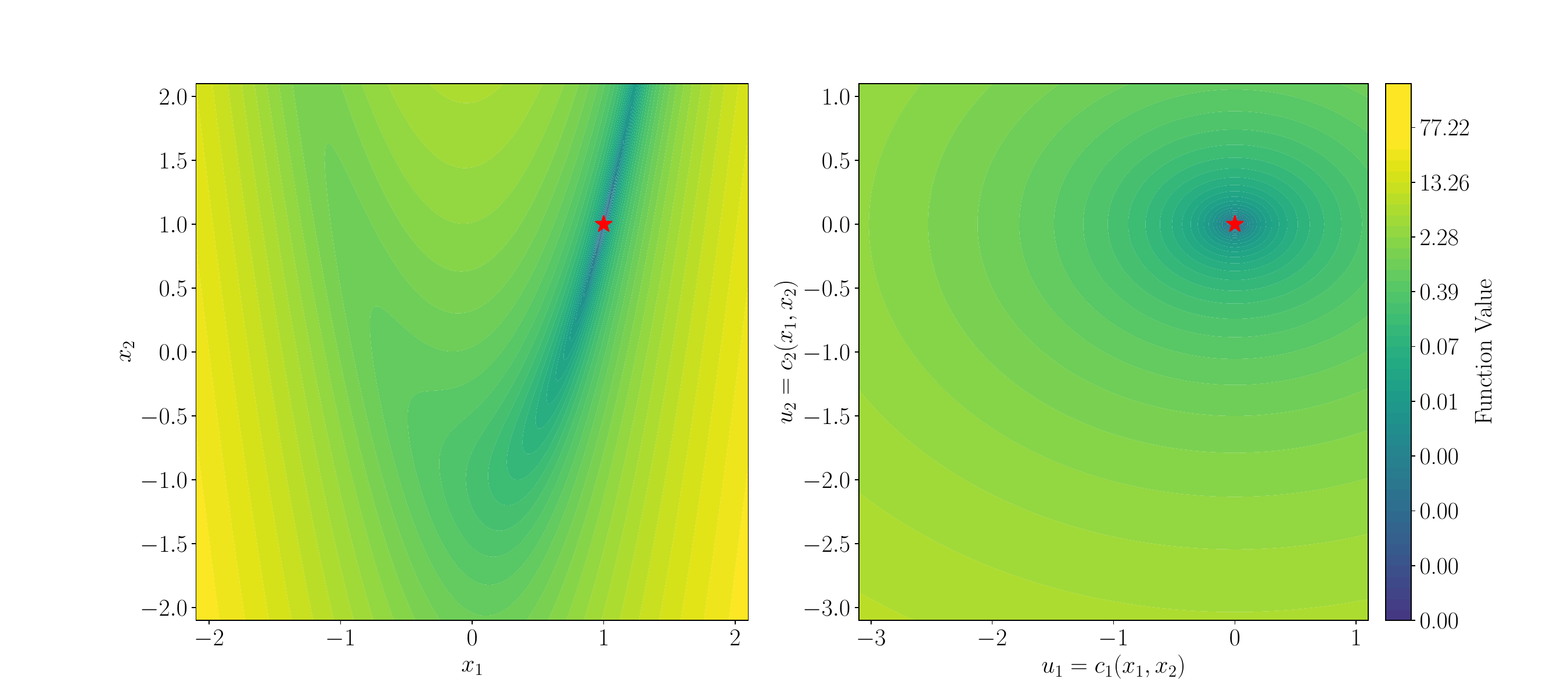}
 	\label{fig:chebyshev_plot}
 	\includegraphics[width=\textwidth]{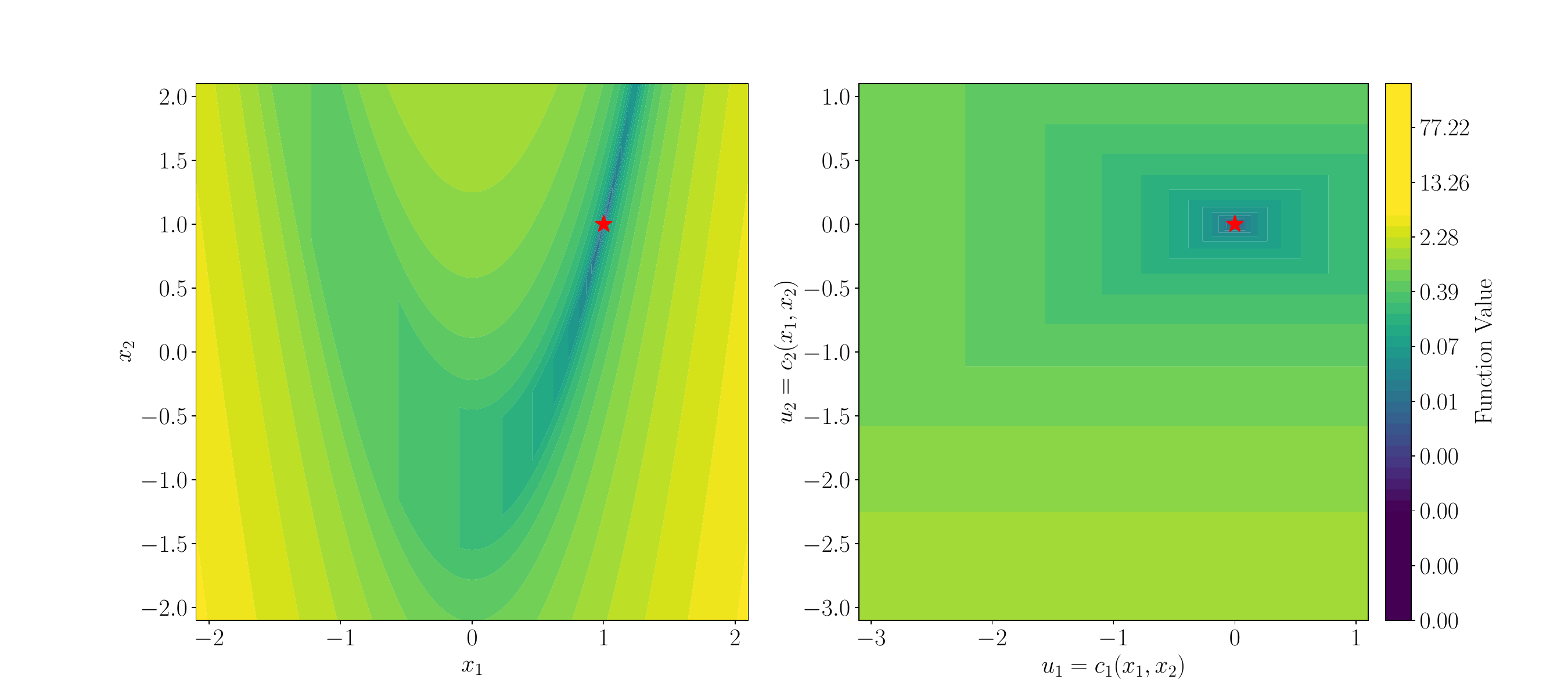}
 	\label{fig:chebyshev_plot_max}
 	\caption{The contour plots of the functions $F(x) = \frac{1}{4} (x_1 - 1)^2 + \frac{1}{2} (2 x_1^2 - x_2 - 1)^2$ (top), and $F(x) = \max\left\{\frac{1}{4} | x_1 - 1 | , \frac{1}{2} | 2 x_1^2 - x_2 - 1 | \right\}$ (bottom), $x = (x_1, x_2)^{\top}$ The left plots present the contour plots in the original space $\cX$ and the right plots illustrate the reformulated space $\cU$. The red star denotes the global minimum.  }
 \end{figure}
 
	\subsection{Minimizing posinomial functions~\cite{boyd2007_GP_tutorial,duffin1967geometric}} For power control in communication systems and optimal doping profile problems \cite{boyd2007_GP_tutorial,duffin1967geometric}, one often needs to minimize posinomial functions
 $F(\cdot) : \R_{++}^d \rightarrow \R$ of the following form
	$$
	F(x) = \sum_{k = 1}^{K} b_k x_1^{a_{1k}} \cdots x_d^{a_{d k}}, 
	$$
	where $b_k > 0$ and $a_{i k } \in \R$ for all $k = 1, \cdots, K$, $i = 1, \cdots, d$. The function $F(\cdot)$ is non-convex but admits a convex reformulation via a variable change $u= c(x) := [\log(x_1),\ldots,\log(x_d)]^\top$. The convex reformulation is of the form
	$$
	H(u) = F(c^{-1}(u))=\sum_{k = 1}^{K} b_k e^{a_{1k} u_1 } \cdots e^{a_{d k} u_d} , 
	$$
	where $H(\cdot)$ is convex. 
 
    We can verify that the above problem is hidden convex on a convex compact set $\cX \subset \R_{++}^d$. {\color{black} Specifically, notice that for any $i = 1,\ldots, d$ and $x\in \cX$, the function $\log(x_i)$ is $(\max_{x \in \cX} x_i)^{-1}$-strongly monotone. Therefore, condition \eqref{C2} holds with $\mu_c = (\max_{x \in \cX} \norm{x}_{\infty})^{-1}$.\footnote{We remark that in this example the parameter $\mu_c$ depends on the upper limit of the variables $x$. On the other hand, the lower limit, $\min_{1\leq i \leq d} \min_{x\in \cX} x_i$, influences the diameter $D_{\mathcal U}$ of the set $\mathcal U = c(\cX)$. This quantity will also frequently appear in the complexity guarantees in \Cref{sec:subgradient_method,sec:PSGD,sec:PSGDM}. }
    }

		\subsection{System level synthesis in optimal control \cite{anderson_2019_SLS}}
		Consider a linear time-varying system
		$$
		x(t+1) = A_t \, x(t) + B_t \, u(t) + w(t) , \qquad t = 0, \ldots, T , 
		$$
		where $x(t) \in \mathbb{R}^n$ is a state, $u(t) \in \mathbb{R}^p$ is a control input, and $w(t) \in \mathbb{R}^n$ is an exogenous disturbance process, and $x(0), w(t) \sim \mathcal{N}(0, \Sigma)$ are independent for $t=0,\ldots, T$. 
  Matrices $A_t \in \R^{n\times n}$ and $B_t \in \R^{n\times p}$ determine the system dynamics.  
		Define $\mathbf{x}=( x(0), \ldots, x(T) )^{\top}$, $\mathbf{u}=( u(0), \ldots, u(T) )^{\top}$, $\mathbf{w}=( x(0), w(0), \ldots, w(T-1) )^{\top}$, and consider a time-varying controller of the form $u(t) = \sum_{i = 0}^{t} K(t, t - i) x(i)$, which depends on a control matrix
		$$
		\mathbf{K}=\left[\begin{array}{cccc}
			K(0,0) & & & \\
			K(1,1) & K(1,0) & & \\
			\vdots & \ddots & \ddots & \\
			K(T, T) & \cdots & K(T, 1) & K(T, 0)
		\end{array}\right].
		$$ 
  The goal of the system level synthesis is to find a control policy to minimize some loss functions, e.g., quadratic in $\mathbf x$ and $\mathbf u$:
		$F(\mathbf K) := \Exp{ \mathbf x^{\top} \mathcal{Q} \mathbf{x} +  \mathbf u^{\top} \mathcal{R} \mathbf{u}  } , $ where $\mathcal{Q} = \text{diag}(Q_0, \ldots, Q(T)) \succeq 0$ and $\mathcal{R} = \text{diag}(R_0, \ldots, R(T)) \succeq 0 $.
 
		Despite the fact that $F(\cdot)$ is  convex in both $\mathbf{x}$ and $\mathbf{u}$, it is non-convex in the decision variable $\mathbf{K}$. Nevertheless, it admits a convex reformulation \cite{anderson_2019_SLS}  of the form 
		$$
		\min_{\Phi_{\mathbf{x}}, \Phi_{\mathbf{u}} } H(\Phi_{\mathbf{x}}, \Phi_{\mathbf{u}} ) , \quad \text{s.t. } \mathbf{M} \left[\begin{array}{c}
			\Phi_{\mathbf{x}} \\ \Phi_{\mathbf{u}} 
		\end{array}\right] = I, \quad \Phi_{\mathbf{x}} , \Phi_{\mathbf{u}} \text{ are lower-block-triangular},
		$$
		where $\Phi_{\mathbf{x}}, \Phi_{\mathbf{u}} \in \R^{(T+1)\times (T+1)}$ are the new variables, $H(\cdot)$ is a strongly-convex function of $\Phi := (\Phi_{\mathbf{x}}, \Phi_{\mathbf{u}})$. $\mathbf{M} \in \R^{(T+1)\times (T+1)} $ is a deterministic matrix, which depends on matrices $A_t$, $B_t$, $t = 0, \ldots, T-1$, and $I$ is the identity matrix. Moreover, there exists a bijection between variables $\mathbf{K}$ and $\Phi$ subject to the constraints of the reformulated problem. The (inverse of the) map $c(\cdot)$ is given by $\mathbf{K} = c^{-1}(\Phi) := \Phi_{\mathbf{u}} \Phi_{\mathbf{x}}^{-1} $ \cite{anderson_2019_SLS}. {\color{black} Therefore, we recognize the hidden convex structure \eqref{problem:original_composition} with invertible $c(\cdot)$.\footnote{{\color{black} Note that, however, it remains elusive whether the condition \eqref{assumption:g_inv} holds across the entire domain. It is possible that a sufficient condition such as \eqref{eq:C2_sufficient} can be verified on the trajectory of some optimization methods instead. }}}
		A number of other problems in optimal control also admit suitable convex reformulations. We refer readers to \cite{boyd1994linear,sun2021learning} for more examples.

	\subsection{Convex reinforcement learning~\cite{zhang2020variational}}
	Convex reinforcement learning (RL) problem generalizes the classical RL setting. It bases on a discounted Markov Decision Process~$\mathbb M(\mathcal{S}, \mathcal{A}, \mathcal{P}, H, \rho, \gamma)$, where $\mathcal{S}$ and $\mathcal{A}$ denote the (finite) state and action spaces respectively, $\mathcal{P}: \mathcal{S}\times\mathcal{A} \to \Delta(\mathcal{S})$ is the state transition probability kernel (where $\Delta(\mathcal{S})$ denotes the distribution over $\mathcal{S}$), $\rho$~is the initial state distribution and $\gamma \in (0,1)$ is the discount factor.  A stationary policy~$\pi: \mathcal{S} \to \Delta(\mathcal{A})$ maps each state~$s \in \mathcal{S}$ to a distribution~$\pi(\cdot|s)$ over the action space $\mathcal{A}$. The set of all (stationary) policies is denoted by~$\Pi$\,.
	At each time step~$h \in \mathbb{N}$ in a state~$s_h \in \mathcal{S}$, the RL agent chooses an action~$a_h \in \mathcal{A}$ with probability~$\pi(a_h|s_h)$ and the environment transitions to a state~$s_{h+1}$ with probability~$\mathcal{P}(s_{h+1}|s_h, a_h)\,.$ We denote by~$\mathbb P_{\rho,\pi}$ the probability distribution of the Markov chain~$(s_h,a_h)_{h \in \mathbb{N}}$ induced by the policy~$\pi$ with an initial state distribution~$\rho$. 
	For any policy~$\pi \in \Pi$, we define the state-action occupancy measure 
	\begin{equation}
		\label{eq:s-a-occup-measure}
		\lambda^{\pi}(s,a) := \sum_{h=0}^{+\infty} \gamma^h \mathbb P_{\rho,\pi}(s_h = s, a_h = a)\, \qquad \text{for all} ~ a \in \mathcal{A}, \, s \in \mathcal{S} \,.
	\end{equation}
 The set of such state-action occupancy measures is denoted by $\cU := \{ \lambda^{\pi} : \pi \in \Pi\}\,.$
 
 Different from the classical RL,  convex RL considers a general (convex) utility function $H:\cU \to \R$ that maps the state-action occupancy measure to a cost and aims to find a policy that minimizes the cost
    \begin{equation}\label{eq:convex_RL_obj}
		\min_{\pi \in \Pi} F(\pi) := H(\lambda^{\pi}) .
	\end{equation}
Notice that $F(\cdot)$ is not convex in $\pi$ in general. However, for several commonly used utility functions, $H(\cdot)$ exhibits convexity in the occupancy measure $\lambda^\pi$. For standard RL, $H(\lambda^{\pi}) = r^\top \lambda^{\pi}$ is linear in $\lambda^{\pi}$, where $r$ is the reward vector. For the pure exploration setting, focused on fully exploring the transitions in the environment, $H(\lambda^{\pi})$ represents the negative entropy of $\lambda^{\pi}$, which is also convex ~\cite{zhang2020variational}. For the imitation learning where the objective is to imitate the expert's behavior given their sampled trajectories, $H(\lambda^{\pi})$ denotes the KL-divergence between $\lambda^{\pi}$ and the state-action occupancy measure learned from the expert's sampled trajectories, which is also convex~\cite{zhang2020variational}. Thus, the convex RL problem belongs to the hidden convex class with $\cX = \Pi$ and $c(x) = \lambda^{\pi}$ (with $x = \pi$). Under mild assumptions on the initial distribution $\rho$, the constant $\mu_c > 0 $ can be estimated, {\color{black} see e.g., Proposition H.1. in \cite{zhang2020variational}}. Note that in convex RL, we can control $\lambda^\pi$ only implicitly by changing the policy $\pi$. The exact computation of the transformation map and its inverse requires the knowledge of the state transition probability kernel and can be computationally expensive.

\subsection{Revenue management and inventory control~\cite{chen2022efficient,chen2022network}}
Consider a booking limit control in a passenger network revenue management problem. The goal is to maximize the revenue by finding an optimal booking limit threshold for each demand class, e.g., flying from New York to Seattle with economy class.  Such a problem forms a two-stage stochastic programming such that
\begin{equation}
	\label{eq:network_revenue_management}
	\begin{aligned}
		\min_{x\in[0,D]^d} &\quad F(x):= \mathbb{E}_{\xi}[ r^\top (x\wedge \xi) - \mathbb{E}_{\eta} \Gamma(x\wedge\xi,\eta) ] \\
		\mathrm{where} & \quad \Gamma(x\wedge \xi, \eta) = \min_{0\leq w\leq x\wedge \xi} \{l^\top (x\wedge\xi-w)\mid Aw\leq \eta\},  
	\end{aligned}
\end{equation}
where $d$ denotes the number of demand classes in the airline networks, $x \in \R^d$ is the booking limit control threshold for each demand class, $\xi$ is the random demand vector (of the same dimension as $x$) during the reservation stage, $x\wedge \xi$ denotes the number of reservations accepted, and $r^\top (x\wedge \xi)$ denotes revenue collected during the reservation stage with $r \in \R^d$ being the price vector. In the service stage, $\Gamma(x\wedge\xi,\eta)$ denotes the penalty on the airline companies when there are $x\wedge \xi$ number of reservations with plane seats capacity $\eta$ that is random,  $w $ is the actual number of passengers that can get on the plane, $l$ is the penalty vector for declining passengers with reservation to get on the plane. Notice that $F$ is non-convex in $x$ due to the truncation between $x$ and $\xi$. When $\xi$ admits component-wise independent coordinates {\color{black} and $x < \text{esssup} \, \xi$}, this problem admits a convex reformulation via a variable change \cite{feng2018supply,chen2022efficient}, i.e., $u= c(x) =\mathbb{E}_\xi [x\wedge\xi]$. {\color{black} Note that the upper bound on the booking limit threshold $D < \text{esssup} \, \xi$, the largest possible demands, naturally holds in revenue management and implies that Condition C.2 holds.} Note that comparing to previous applications, the transformation function involves unknown distribution and thus is not explicitly known. 

For more examples of hidden convex problems in operations research, we refer readers to \cite{feng2018supply} about supply chain management and \cite{chen2022network,miao2021network} about revenue management.

	\section{Properties of Hidden Convex Optimization}\label{sec:prop_HC}
In this section, we provide key properties of hidden convex problems and discuss its connections with gradient dominated function classes.    
	\subsection{Globally optimal solution}
The following proposition suggests that every stationary point of a hidden convex function is a global minima. 

\begin{proposition}\label{prop:global_opt}
		Let $F(\cdot)$ be weakly convex and hidden convex on $\cX$ with $\bar{x} \in \cX$ being its stationary point. If the map $c(\cdot)$ is differentiable at $\bar{x}$
        , then $\bar{x}$ is a global minimum, i.e., $F(\bar{x}) \leq F(x)$ for any $x \in \cX$.
	\end{proposition}
	\begin{proof}
		By the definition of a stationary point and the chain rule \cite{rockafellar2009variational} (Theorem 10.49), we can write
		\begin{equation}\label{eq:chain_rule}
		0 \in \partial_x (F + \delta_{\cX})(\bar{x}) = \nabla c(\bar{x}) \rb{ \partial_u H(\bar{u}) + \partial_u \delta_{\cU}(\bar{u}) }, 
		\end{equation}
		where $\bar{u} = c(\bar{x})$. As the map $c(\cdot)$ is invertible {\color{black} with a Lipschitz continuous inverse by \eqref{C2}, then its Jacobian $\nabla c(x)$ is invertible at $\bar x$ (see e.g., Corollary 3.3.\,in \cite{LAWSON2020123913})}. 
  Therefore, \eqref{eq:chain_rule} implies that $0 \in \partial_u H(\bar{u}) + \partial_u \delta_{\cU}(\bar{u})$. Since function $H(\cdot)$ is convex, by the sufficient optimality condition, $\bar{u}$ is a globally optimal solution, i.e., $H(\bar{u}) \leq H(u)$ for any $u \in \cU$. As a result, we have
		$F(\bar{x}) = H(\bar{u}) \leq H(u) = F(x)$  \text{for any } $x \in \cX$.
	\end{proof}	
	Note that a similar result appeared in {\color{black} Theorem 4.2 of} \cite{zhang2020variational} under additional smoothness assumption on $c(\cdot)$. The above proof is much simpler and does not require smoothness. 
 
	\subsection{Connections with gradient dominated functions}
	It is natural to ask what is the connection between hidden convex problems we consider and previously studied gradient dominated functions \cite{absil2005convergence,attouch2009convergence,attouch2010proximal,frankel2015splitting} that can also be used to ensure the global convergence of (sub-)gradient methods. 
   	{\color{black} We will distinguish between the weak \eqref{eq:HC_1KL} and the strong \eqref{eq:HSC_2KL} variants of gradient domination condition. It turns out that hidden convexity implies the weak gradient domination and hidden strong convexity implies its strong version in the sense specified below. 
    

 \begin{proposition}\label{le:HC_KL}
Let $F(\cdot)$ be weakly convex and hidden convex with modulus $\mu_c$ on $\cX$, the map $c(\cdot)$ be continuously differentiable on $\cX$. 

$(i)$ If the set $\cU = c(\cX)$ is bounded with diameter $D_{\cU}$, then
  \begin{equation}\label{eq:HC_1KL}
		\inf_{s_x \in \partial (F + \delta_{\cX})(x)} \norm{ s_x } \geq \frac{\mu_c}{D_{\cU}} \rb{ F(x) -  F^* } \qquad \text{for all } x \in \cX .
  \end{equation}

$(ii)$ If $F(\cdot)$ is  $(\mu_c, \mu_H)$-hidden strongly convex, then 
  \begin{equation}\label{eq:HSC_2KL}
		\inf_{s_x \in \partial (F + \delta_{\cX})(x)} \sqnorm{ s_x } \geq 2 \mu_H \mu_c^2 \rb{F(x) -  F^* } \qquad \text{for all } x \in \cX .
  \end{equation}
 \end{proposition}
    
}
\begin{proof}
{\color{black} 
$(i)$ Since the map $c(\cdot)$ is invertible with a Lipschitz continuous inverse by \eqref{C2}, its Jacobian $\nabla c(x)$ is invertible for all $x\in \cX$ (Corollary 3.3.\,in \cite{LAWSON2020123913}). Thus, $\norm{(\nabla c(x))^{-1}}_{\text{op}}$ is bounded on a compact set $\cX$ and we can apply a global inverse function theorem due to Hadamard, see e.g., Theorem 1 in \cite{ivanov2023hadamard}. It implies that the inverse map $c^{-1}(\cdot)$ is also continuously differentiable on $\cU$. Therefore, we can show that condition \eqref{eq:C2_sufficient} holds, i.e., $\norm{\nabla c^{-1}(u) }_{\text{op}} \leq 1/\mu_c$ for all $u\in \cU$. By the sub-gradient inequality applied to $H + \delta_{\cU}$, we have for any $x \in \cX$, $x^* \in \cX^*$, $u = c(x)$, $u^* = c(x^*)$ and any $s_u \in \partial (H + \delta_{\cU})(u)$ that
\begin{eqnarray*}
F(x) - F(x^*) & = & H(u) - H(u^*) \\
&\leq& \langle s_u , u - u^* \rangle \\ 
&\leq& \| s_u \| \cdot \| u - u^* \| \\
&=& \| \nabla c^{-1}(u) \, s_x \| \cdot \| u - u^* \| \\
&\leq& \| \nabla c^{-1}(u) \|_{\text{op}} \cdot \|  s_x  \| \cdot \| u - u^* \| \\
&\leq& \frac{D_{\mathcal U}}{\mu_{c}} \| s_x \| ,
\end{eqnarray*}
where the first inequality uses convexity of $H$, the second equality uses the chain rule with $s_x \in \partial (F + \delta_{\cX})(x) $. Taking the infimum over $s_x$ concludes the proof of the first statement. 

$(ii)$ First, we will show that \ref{C2} implies $\norm{\nabla c(x) \, z} \geq \mu_c \norm{z}$ for any $z\in\R^d$. Let $x\in \operatorname{int} \cX$, then there exists $\alpha > 0$ such that for any $\tau \in [0, \alpha]$ the point $x+\tau z\in \cX$. Thus, we can write 
$$\left(\int_0^{\alpha} \nabla c(x + \tau z) d \tau \right)z = c(x+\alpha z) - c(x). $$
Using \eqref{assumption:g_inv}, we have
$$\left\| \left(\int_0^{\alpha} \nabla c(x + \tau z) d \tau \right)z \right\| = \|c(x+\alpha z) - c(x)\|\geq\mu_{c}\alpha\|z\|.$$
Dividing by $\alpha > 0$ and taking the limit $\alpha \rightarrow 0$, we arrive at $\|\nabla c(x) z \| \geq \mu_{c}\|z\|$ for all $z\in \mathbb R^d$. If $x\in \partial \cX$, then there exists a sequence $\left\{x_t\right\}_{t\geq 0} \subset \operatorname{int} \cX$ such that $\lim_{t\rightarrow \infty} x_t =  x$. By continuity of $\|\nabla c(x) z \|$ in $x$, we have $\|\nabla c(x) z \| = \lim_{t\rightarrow\infty} \|\nabla c(x_t) z \| \geq \mu_{c}\|z\|$.

Next, by strong convexity of $H+ \delta_{\cU}$, we have for any $s_u \in \partial (H + \delta_{\cU})(u)$ that 
$$
H(v) \geq H(u) + \langle s_u , v - u \rangle + \frac{\mu_H}{2} \sqnorm{v - u} \qquad \text{for any } u, v \in \cU . 
$$

Minimizing both sides over $v
\in \cU$, we get 
\begin{equation}\label{eq:KL}
 \sqnorm{s_u} \geq 2 \mu_H \rb{ H(u) - H(u^*) }.
\end{equation}

Finally, using the chain rule, we have for any $u = c(x)$
	\begin{eqnarray*} 
		 \sqnorm{s_x} & = & \sqnorm{ \nabla c(x)  \, s_u  } \geq \mu_c \sqnorm{ s_u  } 
			\geq 2 \mu_H \mu_c^{2} \rb{F(x) - F(x^*) },
	\end{eqnarray*}
where in the last inequalities, we used the fact that $\norm{\nabla c(x) \, z} \geq \mu_c \norm{z}$ for any $z \in \R^d $, and \eqref{eq:KL}. It only remains to take the infimum over $s_x$. 
}
 \end{proof}

{\color{black} 
In certain special cases, one can establish convergence rates of gradient methods using such gradient dominance conditions. In particular, if there is no noise and the sequence $\{x^t\}_{t\geq 0}$ satisfies a (strong) descent property, the results from \cite{absil2005convergence} indicate global convergence rates. However, given that in general the sub-gradient method does not have a descent property, this implication seems limiting even in the deterministic case. When $F$ is additionally smooth, existing theory of Projected SGD under gradient dominated conditions is limited either to unconstrained case \cite{fontaine2021convergence,KL_PAGER_Fatkhullin} with the strong gradient dominance condition \eqref{eq:HSC_2KL} \cite{karimi2016linear,li2018simple}, or requires to replace \eqref{eq:HC_1KL} with a weak Proximal-P{\L} inequality \cite{csiba2017global}, which is different from \eqref{eq:HC_1KL}. Unfortunately, the precise connection between \eqref{eq:HC_1KL} and weak Proximal-P{\L} appears difficult to establish and is beyond the scope of the current work. 

To circumvent the above mentioned technical difficulties, in what follows, we will not use the result of \Cref{le:HC_KL}, but develop a direct convergence proof utilizing hidden convexity. Our proof technique allows to handle non-smooth, stochastic and constrained cases automatically without selecting a proper variant of P{\L} condition (based on algorithm) and attempting to verify it.

}

	\subsection{Key inequalities for analysis of gradient methods}
	\label{subsec:key_ineq}
	The following observations are key tools for deriving global convergence under hidden convexity.
	\begin{proposition}\label{prop:HC}
		Let $F(\cdot)$ be hidden convex with $\mu_c>0, \mu_H\geq0$. For any $\alpha \in [0,1]$, $x^* \in \cX^*$ and $x\in \cX$ , define $x_{\alpha} := c^{-1} \rb{(1-\alpha) c(x) + \alpha c(x^*)}$. Then
		\begin{eqnarray}\label{prop:F_x_alpha_strong_conv}
			F(x_{\alpha}) \leq  (1- \alpha) F(x) + \alpha  F(x^*)  -   \fr{(1-\alpha) \alpha  \mu_H}{2}  \sqpnorm{c(x) - c(x^*) } ,
		\end{eqnarray}	\begin{eqnarray}\label{prop:x_alpha_x_bound}
			\pnorm{ x_{\alpha}-x } \leq \fr{\alpha }{\mu_c} \pnorm{c(x) - c(x^*) } .
		\end{eqnarray}
		
	\end{proposition}
	
	\begin{proof}
		By the (strong) convexity of $H(\cdot)$ and the convexity of $\cU$, we have
		\begin{eqnarray*}
			F(x_{\alpha}) &=& F( c^{-1} \rb{(1-\alpha) c(x) + \alpha c(x^*)} ) \\
			& = & H( (1-\alpha) c(x) + \alpha c(x^*) ) \\ 
			& \leq  &  (1-\alpha) H( c(x) ) + \alpha H( c(x^*) ) -   \fr{(1-\alpha) \alpha  \mu_H}{2} \sqpnorm{c(x) - c(x^*) }\\
			& =  &  (1- \alpha) F(x)  +  \alpha F(x^*) -  \fr{(1-\alpha)  \alpha   \mu_H}{2} \sqpnorm{c(x) - c(x^*) }.
		\end{eqnarray*}
  where the inequality uses the fact that $\cU$ is a convex set and that $(1-\alpha) c(x) + \alpha c(x^*) \in \cU$ for any $x \in \cX$. By definition of $x_{\alpha}$ and \eqref{assumption:g_inv}, we derive
		$$
		\pnorm{ x_{\alpha}-x } = \pnorm{ c^{-1} \rb{(1-\alpha) c(x) + \alpha c(x^*)} - c^{-1}(c(x)) } \leq \fr{1}{\mu_c} \pnorm{\alpha(c(x) - c(x^*)) } .
		$$
		\end{proof}

	\section{Stochastic Subgradient Method}\label{sec:subgradient_method}

  In this section, we show how \Cref{prop:HC} can be used to analyze convergence of the projected stochastic subgradient method (SM) as described in \eqref{eq:subgradient_method} in the non-smooth setting.

We first make the following assumptions.
 
	\begin{enumerate}
		\myitem{A.1}\label{A1} $F(\cdot)$ is $\ell$-weakly convex on a closed, convex set $\cX$. 
		\myitem{A.2}\label{A2} We have access to a stochastic sub-gradient oracle of $F(\cdot)$ at any $x\in\cX$, which outputs a random vector $g(x, \xi)$ such that $\Exp{ g(x, \xi) } \in \partial F(x),$ where $\partial F(x)$ is the sub-differential set of $F(\cdot)$ at $x$. Moreover, there exists $G_F > 0$ 
		$$
		\Exp{ \sqnorm{ g(x, \xi)} } \leq G_F^2  \qquad \text{for any } x \in \cX .
		$$
	\end{enumerate}	

 The above assumptions are standard and appear frequently in non-smooth optimization \cite{davis2019proximally,zhang_he_2018_msgd,davis2018stochastic_high_order_growth,davis2019stochastic}. 
 Weak convexity is known to be a much weaker condition than smoothness \cite{davis2019stochastic}.  
 Notably, in the context of our hidden convexity (C.1. and C.2.), weak convexity is not restrictive and comes for free from the Lipschitz continuity of $H(\cdot)$ and the smoothness of the transformation function $c(\cdot)$. Specifically, if $H: \cU \rightarrow \R$ is convex and $G_H$-Lipschitz continuous on $\cU$ and $c: \cX \rightarrow \cU$ is $L_c$-smooth, then it can be shown that the composition $F(x) = H(c(x))$ is $\ell$-weakly convex with $\ell := G_H L_c$; see e.g. Proposition 2.2(c) in \cite{zhang_he_2018_msgd}. 
 In the absence of smoothness, the second assumption on bounded second moment of the (stochastic) sub-gradients is typical even in convex case. Later in Section~\ref{sec:PSGD}, we show this assumption can be further relaxed to bounded variance in the smooth setting.

  Let $x^* \in \cX^*$, and $\Phi := F + \delta_{\cX}$. We define the Lyapunov function
	$$\Lambda_t := \Exp{ \Phi_{1/\rho}(x^t) - F(x^*) } , $$
	where $\Phi_{1/\rho}$ is the Moreau envelope of $\Phi$. Notice that $\Lambda_t \geq 0$ for any $t\geq 0$ and $\Lambda_t = 0$ if and only if $x^t \in \cX^*$. 

 Before stating the main result, we recall the following useful lemma from~\cite{davis2019stochastic} that controls the distance between one step of the SM, $x^{t+1}$, and one step of the proximal point method, $\hat x^t$. We include its proof in~\Cref{sec:technical_lemma} for completeness. 
 	
\begin{lemma}[Lemma 3.3 in \cite{davis2019stochastic}]\label{le:prox_grad_prox_bound_nonsmooth}
	Let \ref{A1}, \ref{A2} hold, and $\rho = 2 \ell$, $\stepsize \leq 1/\rho$ . Then for all $t \geq 0$, we have 
	$
    \mathbb{E}\left[ \sqnorm{ x^{t+1}-\hat{x}^t } \mid x^t\right] \leq(1-\stepsize \rho)\sqnorm{x^t-\hat{x}^t } + 4 G_F^2 \stepsize^2 .
	$
\end{lemma}	

 The next theorem is the essential step for establishing the global convergence of SM  in \cref{thm:subgrad_C,thm:subgrad_SC}.
	
	\begin{theorem}\label{le:descent_like_subgrad}
		Let \ref{C1}, \ref{C2}, \ref{A1} and \ref{A2} hold with $\mu_H\geq 0$. Set $\rho = 2 \ell$, $\stepsize \leq \frac{1}{2\ell}$. Define 
  $\hat x^t := \operatorname{prox}_{ \Phi / \rho }(x^t)$. Then for any $0 < \alpha \leq \stepsize  \ell  $ and $t\geq 0$
		\begin{eqnarray*}
			 \Lambda_{t+1}  &\leq&  (1- \alpha) \Lambda_t   + \rb{  \fr{ 3 \alpha^{2} }{ 2 \mu_c^2 \stepsize } -  \fr{(1-\alpha)\alpha  \mu_H}{2} } \Exp{\sqnorm{ c(\hat{x}^t) - c(x^*) } } + 8 \ell  \stepsize^2 G_F^2 . 
		\end{eqnarray*}
	\end{theorem} 
	\begin{proof}
		By the definition of $\hat x^{t+1}$, we have for any $z \in \cX$
		$$
		\begin{aligned}
			\mathbb{E}\left[\Phi_{1/\rho} \left(x^{t+1}\right)\right] & =\mathbb{E} \left[\Phi\left(\hat{x}^{t+1}\right)+\frac{\rho}{2} \sqnorm{ \hat{x}^{t+1}-x^{t+1}} \right] \\
			& \overset{(i)}{\leq} \mathbb{E}\left[\Phi\left(z \right) + \frac{\rho}{2}\sqnorm{z  -x^{t+1}}\right] \\
			& \overset{(ii)}{\leq} \mathbb{E}\left[\Phi\left(z \right) +  \rb{1 + s } \frac{\rho}{2}\sqnorm{\hat x^{t} -x^{t+1}} + \rb{1 + \frac{1}{s} } \frac{\rho}{2}\sqnorm{\hat x^{t} -z }\right] \\
			& \overset{(iii)}{\leq} \mathbb{E}\left[\Phi\left( z \right) + \rb{1 + s } (1-\stepsize \rho) \frac{\rho}{2}\sqnorm{ \hat{x}^t-x^t }\right] \\
			& \qquad + \rb{1 + \frac{1}{s} } \frac{\rho}{2}\Exp{\sqnorm{\hat x^{t} - z }} + 2 \rb{1 + s } \rho \stepsize^2 G_F^2 , 
		\end{aligned}
		$$
		where in $(i)$ we use the optimality of $\hat x^{t+1} $, $(ii)$ follows from Young's inequality for any $s> 0$, and in $(iii)$ we apply the result of \cref{le:prox_grad_prox_bound_nonsmooth}.
		We now select $s = \stepsize \rho / 2$, which guarantees $\rb{1 + s } (1-\stepsize \rho) \leq 1 - \stepsize \rho / 2$, $1+s \leq 2$, and $1+1/s \leq 3 / (\stepsize \rho) $. Thus 
		\begin{eqnarray*}
			\mathbb{E}\left[\Phi_{1/\rho}\left(x^{t+1}\right)\right] &\leq& \mathbb{E}\left[F\left( z \right) +  \rb{1 - \frac{\stepsize \rho}{2} } \frac{\rho}{2} \sqnorm{\hat{x}^t-x^t } \right] + \frac{3}{2 \stepsize }\Exp{\sqnorm{\hat x^{t} - z }} + 4 \rho \stepsize^2 G_F^2 . 
		\end{eqnarray*}
		
		We are now ready to utilize the properties of hidden convex functions to bound $F(z)$ and $\sqnorm{ \hat{x}^t - z }$ for some specific choice of $z \in \cX$. 
		By~\Cref{prop:HC}, we have for $z = \hat x_{\alpha}^t := c^{-1}((1-\alpha) c(\hat x^t) + \alpha c(x^*))$ 
		$$
		F(
		z
  ) \leq (1- \alpha) F(\hat x^t) +  \alpha F(x^*) -   \fr{(1-\alpha)\alpha  \mu_H}{2} \sqnorm{ c(\hat x^t) - c(x^*) } ,
		$$
		$$
		\sqnorm{z
  - \hat x^t } \leq \fr{ \alpha^{2} }{ \mu_c^2  } \sqnorm{ c(\hat x^t) - c(x^*) }.
		$$
		Combining three inequalities  above, we have 
		\begin{eqnarray*}
			\Exp{ \Phi_{1/\rho}(x^{t+1}) }  &\leq&  (1- \alpha) \Exp{ F(\hat x^t) } +  \alpha F(x^*) + \rb{1 - \frac{\stepsize \rho}{2} } \frac{\rho}{2} \Exp{ \sqnorm{\hat{x}^t-x^t } }  + 4 \rho \stepsize^2 G_F^2 \\
			&& \qquad + \rb{  \fr{ 3 \alpha^{2} }{ 2 \mu_c^2 \stepsize } -  \fr{(1-\alpha)\alpha  \mu_H}{2} } \Exp{\sqnorm{ c(\hat x^t) - c(x^*) } } \\
   &\leq&  (1- \alpha) \Exp{ \Phi_{1/\rho}(x^t) } +  \alpha F(x^*)  + 4 \rho \stepsize^2 G_F^2 \\
			&& \qquad + \rb{  \fr{ 3 \alpha^{2} }{ 2 \mu_c^2 \stepsize } -  \fr{(1-\alpha)\alpha  \mu_H}{2} } \Exp{\sqnorm{ c(\hat x^t) - c(x^*) } } , 
		\end{eqnarray*}
  where the last inequality holds since $1-\frac{\stepsize\rho}{2} \leq 1-\alpha $ (by the choice $\alpha \leq \stepsize \ell$, $\rho = 2 \ell$) and recognizing $\Phi_{1/\rho}(x^t)$. Subtracting $F(x^*)$ from both sides, we conclude the proof. 
	\end{proof}
	
	\subsection{Hidden Convex Setting}
 We first demonstrate the convergence rate of SM in the hidden convex setting. 
\begin{tcolorbox}[colback=gray!5!white,colframe=gray!75!black]
	\begin{theorem}\label{thm:subgrad_C}
		Let \ref{C1}, \ref{C2}, \ref{A1}, \ref{A2} hold with $\mu_H = 0$, and the set $\cU$ be bounded by a diameter $D_{\cU}$. 
		Fix $\varepsilon > 0$, and set the step-size in \eqref{eq:subgradient_method} as 
		$\stepsize =  \frac{1}{2 \ell} \cdot \min\left\{ 1 ,  \frac{\mu_c^2  \varepsilon^2}{48 D_{\cU}^2 G_F^2 }  \right\}. $
		Then for $\rho=2 \ell$, we have $\Lambda_T \leq \varepsilon$ after
		$	T = \wt \cO\rb{ \frac{ \ell D_{\cU}^2 }{\mu_c^2 } \frac{1}{\varepsilon}   + \frac{ \ell D_{\cU}^4 G_F^2 }{\mu_c^4  } \frac{1}{\varepsilon^3}  }  $ iterations.
	\end{theorem}
\end{tcolorbox}

	\begin{proof}
	Setting $\mu_H = 0$ in \Cref{le:descent_like_subgrad} and leveraging compactness of $\cU$, we have 
		\begin{eqnarray*}
			\Lambda_{t+1}  &\leq&  (1- \alpha) \Lambda_t  + \fr{ 3 D_{\cU}^2 \alpha^{2} }{ 2 \mu_c^2 \stepsize }  + 8 \ell \stepsize^2 G_F^2 .     
		\end{eqnarray*}
		Unrolling the recursion for $t = 0$ to $t = T-1$, we get
		\begin{eqnarray*}
			\Lambda_{T}  &\leq&  (1 - \alpha)^{T} \Lambda_0  + \fr{ 3 D_{\cU}^2 \alpha }{ 2  \mu_c^2 \stepsize }   + \frac{ 8 \ell \stepsize^2 G_F^2 }{ \alpha } \leq \varepsilon , 
		\end{eqnarray*}
		where the last step holds by setting $\alpha = \min\left\{  \stepsize \ell, \frac{2 \varepsilon \mu_c^2 \stepsize }{9 D_{\cU}^2} ,  \frac{\sqrt{32 \ell} \mu_c G_F \stepsize^{\nfr{3}{2}} }{ \sqrt{3} D_{\cU} }  \right\} $ 
		after 
		$
		T = \frac{1}{ \alpha} \log\rb{\frac{3 \Lambda_0}{\varepsilon}} = \wt \cO \rb{ \frac{\ell D_{\cU}^2}{\mu_c^2  \varepsilon} + \frac{\ell D_{\cU}^4 G_F^2 }{\mu_c^4 \varepsilon^3} } .
		$
	\end{proof}

	We remark that in the absence of smoothness of $F(\cdot)$, the guarantee on $\Lambda_t$ might not necessarily translate to the function value gap $F(x^t)-F(x^*)$. However, with the following corollary we show that the output of SM, $x^T$, is in fact close to an $\varepsilon$-approximate global solution $\hat x^T=\operatorname{prox}_{\Phi / \rho}(x^T)$. 
	
	\begin{corollary}\label{cor:subgrad_C}
	Under the setting of \Cref{thm:subgrad_C}, SM finds a point $x^T \in \cX$, which is close to $\hat x^T$, an $\varepsilon$-global solution of  \eqref{problem:original}. More specifically, it holds that $\Exp{\sqnorm{\hat x^T - x^T}} \leq \varepsilon / \ell$ and $\Exp{F(\hat x^T) - F(x^*)} \leq \varepsilon$ after $T = \wt \cO(\varepsilon^{-3})$.
	\end{corollary}
 \begin{proof}
    The result follows directly from the definition of $\Lambda_T$ and \Cref{thm:subgrad_C}.
 \end{proof}
{\color{black}
 The global convergence of iterates of SM in function values follows immediately from \Cref{cor:subgrad_C}. By the Lipschitz continuity of $F$, we have $\mathbb E F(x^T) - F(x^*) \leq \mathbb E F(\hat x^T) - F(x^*) + G_F \mathbb E \|\hat x^T - x^T\| \leq \varepsilon + G_F \sqrt{\varepsilon \ \ell} $. Taking $\varepsilon$ small, we observe that SM converges globally in function value to arbitrary small accuracy in expectation. In what follows, this result is tightened and improved in several aspects when more structures (hidden strong convexity of $F$ or smoothness of $F$) are available. 
}
		\subsection{Hidden Strongly Convex Setting}
 The following theorem presents a stronger result in the case when $F(\cdot)$ is additionally hidden strongly convex. 
 
\begin{tcolorbox}[colback=gray!5!white,colframe=gray!75!black]
	
		\begin{theorem}\label{thm:subgrad_SC}
		Let \ref{C1}, \ref{C2}, \ref{A1}, \ref{A2} hold with $\mu_H > 0$. Then for $\rho = 2 \ell$ and any $\stepsize \leq \frac{1}{2 \ell}$, $\alpha \leq \min\left\{ \stepsize \ell, \frac{ \stepsize \mu_c^2 \mu_H }{2} \right\}$, we have for all $T \geq 0$
		\begin{eqnarray*}
			\Lambda_T &\leq&  (1- \alpha)^{T} \Lambda_0  + \frac{ 8 \ell \stepsize^2 G_F^2 }{ \alpha }  . 
		\end{eqnarray*}
		Fix $\varepsilon > 0$, and set the step-size in \eqref{eq:PSGD} as 
		$\stepsize = \min\left\{ \frac{1}{2 \ell} , \frac{\mu_c^2  \mu_H \varepsilon}{20 \ell G_F^2 }  \right\}. $
		Then $\Lambda_T \leq \varepsilon$ after
		$	T = \wt \cO\rb{ \frac{ \ell }{\mu_c^2 \mu_H }  + \frac{ \ell G_F^2 }{\mu_c^4  \mu_H^2 } \frac{1}{\varepsilon}  } $ iterations.
	\end{theorem}
	
\end{tcolorbox}

	\begin{proof}
		We invoke \Cref{le:descent_like_subgrad} with $\mu_H > 0$. The choice of $\alpha$ guarantees the coefficient in front of $\Exp{\sqnorm{ c(\hat x^t) - c(x^*) } } $ is non-positive and
		\begin{eqnarray*}
			\Lambda_{t+1}  &\leq&  (1- \alpha) \Lambda_t   + 8 \ell \stepsize^2 G_F^2 .     
		\end{eqnarray*}
		It remains to conclude the proof by unrolling the recursion and setting the step-size accordingly.
	\end{proof}
In the presence of hidden strong convexity, since the optimal $x^* \in \cX^*$ is unique, we can establish a strong convergence of the sequence $\{x^t\}_{t\geq0}$ to $x^*$. 
\begin{corollary}\label{cor:subgrad_SC}
		Let the assumptions of \Cref{thm:subgrad_SC} hold and $x^T$ be the output of the method \eqref{eq:subgradient_method} after $T$ iterations (given by \Cref{thm:subgrad_SC}). Then $ \Exp{ \sqnorm{x^T - x^*} } \leq \rb{ \frac{4}{\mu_H \mu_c^2} + \frac{2}{\ell} } \varepsilon $ . 
	\end{corollary}

	\begin{proof}
		Since $H(\cdot)$ is $\mu_H$-strongly convex on $\cX$, we have 
		\begin{eqnarray}\label{eq:SC_implication}
			F(\hat x^T) - F(x^*) = H(c(\hat x^T)) - H(c(x^*)) \geq \frac{\mu_H}{2}\sqnorm{c(\hat x^T) - c(x^*) } \geq \frac{\mu_H \mu_c^2}{2}\sqnorm{ \hat x^T - x^* }, \notag \\
		\end{eqnarray}
		where the first inequality follows by the first-order characterization of strong convexity and the optimality condition, and the last inequality holds by C.2.  
		
		Recall that $\Lambda_T = \Exp{ F(\hat x^T) - F(x^*)  + \frac{\rho}{2} \sqnorm{\hat x^T - x^T} } $ with $\rho = 2\ell$. Then 
		\begin{eqnarray*}
			 \Exp{ \sqnorm{x^T - x^*} } &\leq& 2 \Exp{ \sqnorm{\hat x^T - x^*} } + 2 \Exp{ \sqnorm{\hat x^T - x^T} } \\
			&\leq& \frac{4}{\mu_H \mu_c^2} \Exp{ F(\hat x^T) - F(x^*) } + 2 \, \Exp{ \sqnorm{\hat x^T - x^T} } \leq \rb{ \frac{4}{\mu_H \mu_c^2} + \frac{2}{\ell} } \varepsilon .
		\end{eqnarray*}
		where the second inequality holds by \eqref{eq:SC_implication} and the last step follows by \Cref{thm:subgrad_SC}. \end{proof}
	
	The above results highlight a notable observation:  although $F(\cdot)$ is non-smooth and non-convex, simple SM converges to a globally optimal solution. This stands in contrast to recent results in general non-smooth non-convex optimization, where more sophisticated (randomized) algorithms are needed to obtain a meaningful solution (e.g., $(\delta, \epsilon)$-Goldstein stationary point) \cite{jordan2023deterministic}. Moreover, it is worth emphasizing the distinctions in the sample complexity results compared to classical findings in convex settings:  \Cref{thm:subgrad_C,thm:subgrad_SC} implies the sample complexities  of $\wt\cO(\varepsilon^{-3})$ and $\wt\cO(\varepsilon^{-1})$ respectively to reach $\Exp{ \Phi_{1/\rho}(x) - F(x^*) } \leq \varepsilon$ for hidden convex and hidden strongly convex problems, whereas in convex and strongly convex settings, the sample complexities are  $\cO(\varepsilon^{-2})$ and $\cO(\varepsilon^{-1})$ respectively to reach $\Exp{F(x) - F^*} \leq \varepsilon$ \cite{nemirovski1983problem}.

	  \section{Projected SGD}\label{sec:PSGD}
In this section, we consider the smooth setting when $F$ is continuously differentiable.  In this case,  SM reduces to Projected SGD (P-SGD): 
\begin{tcolorbox}[colback=white,colframe=gray!75!black]
	\begin{equation}\label{eq:PSGD}
			x^{t+1} = \Pi_{\mathcal{X}}(x^t-\stepsize  \nabla f(x^{t}, \xi^{t}) ) .  
	\end{equation}
\end{tcolorbox}

In particular, we assume that

    \begin{enumerate}
    \myitem{A.1'}\label{A1'} The function $F: \cX \rightarrow \R$ is differentiable on a closed, convex set $\cX$ and its gradient $\nabla F(x)$ is $L$-Lipschitz continuous. 
    
    \myitem{A.2'}\label{A2'} We have access to an unbiased stochastic gradient oracle with bounded variance $\sigma > 0$, i.e. for any $x\in \cX$: 
	$\Exp{\nabla f(x,\xi)} = \nabla F(x)$, and 
	$$\Exp{\sqnorm{\nabla f(x,\xi) - \nabla F(x)}} \leq \sigma^2, $$where expectations are with respect to the random variable $\xi \sim \cD$.  
    \end{enumerate}

Note that Assumption \ref{A2'} on bounded variance is considerably weaker than Assumption \ref{A2} on the bounded second moment of stochastic gradient in previous section. 
	Replacing \cref{le:prox_grad_prox_bound_nonsmooth} with \cref{le:prox_grad_prox_bound} in the proof of \cref{le:descent_like_subgrad} of the previous section, we are able to derive the following results under Assumptions \ref{A1'} and \ref{A2'}. 
	
	\begin{theorem}\label{le:descent_like_Proj_SGD}
		Let  \ref{C1}, \ref{C2}, \ref{A1'}, \ref{A2'} hold with $\mu_H\geq 0$. Set $\rho=4 L$, $\stepsize \leq \frac{2}{9 L}$. Define 
		$\hat x^t := \operatorname{prox}_{ \Phi / \rho }(x^t)$. Then for any $0 < \alpha \leq 2\stepsize  L  $ and $t\geq 0$
		\begin{eqnarray*}
			\Lambda_{t+1}  &\leq&  (1- \alpha) \Lambda_t  + \rho \stepsize^2 \sigma^2  + \rb{  \fr{ 3 \alpha^{2} }{ 2 \mu_c^2 \stepsize } -  \fr{(1-\alpha)\alpha  \mu_H}{2} } \Exp{\sqnorm{ c(\hat{x}^t) - c(x^*) } } . 
		\end{eqnarray*}
	\end{theorem} 

 Using the above result, we provide a refined analysis of Projected SGD in the differentiable setting with smoothness and bounded variance. 

\subsection{Hidden Convex Setting} We start with the hidden convex case.

\begin{tcolorbox}[colback=gray!5!white,colframe=gray!75!black]
	
		\begin{theorem}\label{thm:Proj_SGD_C}
		Let \ref{C1}, \ref{C2}, \ref{A1'}, \ref{A2'} hold with $\mu_H = 0$, and the set $\cU$ be bounded by a diameter $D_{\cU}$. Then for $\rho = 4 L$ and any $\stepsize \leq \frac{2}{9 L}$, $\alpha \leq 2\stepsize  L  $, we have for all $T \geq 0$
		\begin{eqnarray*}
			\Lambda_T &\leq&  (1- \alpha)^{T} \Lambda_0  + \fr{ 3 D_{\cU}^2 \alpha }{ 2 \mu_c^2 \stepsize }   + \frac{ 4 L \stepsize^2 \sigma^2 }{ \alpha }  . 
		\end{eqnarray*}
		Fix $\varepsilon > 0$, and set the step-size in \eqref{eq:PSGD} as 
		$\stepsize = \frac{2}{9 L} \cdot \min\left\{ 1 ,  \frac{\mu_c^2  \varepsilon^2}{12 D_{\cU}^2 \sigma^2 }  \right\}. $
		Then $\Lambda_T \leq \varepsilon$ after
		$	T = \wt \cO\rb{ \frac{ L D_{\cU}^2 }{\mu_c^2 } \frac{1}{\varepsilon}   + \frac{ L D_{\cU}^4 \sigma^2 }{\mu_c^4  } \frac{1}{\varepsilon^3}  }  $ iterations.
	\end{theorem}
	
\end{tcolorbox}
	
	\begin{proof}
 We set $\alpha = \min\left\{ 2 \stepsize L, \frac{2 \varepsilon \mu_c^2 \stepsize }{3 D_{\cU}^2} ,  \frac{\sqrt{8 L} \mu_c \sigma \stepsize^{\nfr{3}{2}} }{ \sqrt{3} D_{\cU} }  \right\} $, then, given \Cref{le:descent_like_Proj_SGD}, the reminder of the proof is similar to the proof of \Cref{thm:subgrad_C} in the previous section. 
	\end{proof}
	

 Similar to \Cref{cor:subgrad_C}, we can show that $x^T$ is close to an $\varepsilon$-global optimal solution. But in the case of smooth $F(\cdot)$, we can also derive a stronger result after applying one (post-processing) step of Projected SGD with mini-batch. We defer this result to \Cref{cor:Proj_SGD_C} in \Cref{sec:minibatching_appendix}. 

\Cref{thm:Proj_SGD_C} implies that in deterministic case when $\sigma^2=0$, the iteration complexity of the gradient method is $\wt \cO(\varepsilon^{-1})$, which coincides with the iteration complexity of Projected GD in the smooth convex setting in terms of $\varepsilon$ (up to a logarithmic factor). However, in the stochastic setting, $\wt \cO(\varepsilon^{-3})$ sample complexity is worse than the well known $\cO(\varepsilon^{-2})$ sample complexity in the convex case \cite{lan2020first}. 
On the other hand, for general smooth nonconvex optimization, Projected SGD is only known to converge to a first-order stationary point (FOSP), i.e., find $x\in \cX$ with $\Exp{\norm{\nabla F(x)}}\leq \epsilon$, with the sample complexity $\cO(\epsilon^{-4})$  \cite{lan2020first,davis2019stochastic}.

\subsection{Hidden Strongly Convex Setting} Similar to the exposition in \Cref{sec:subgradient_method}, we present an improved sample complexity result for hidden strongly convex problems.

\begin{tcolorbox}[colback=gray!5!white,colframe=gray!75!black]
	
\begin{theorem}\label{thm:Proj_SGD_SC}
		Let \ref{C1}, \ref{C2}, \ref{A1'}, \ref{A2'} hold with $\mu_H > 0$. Then for $\rho = 4 L$ and any $\stepsize \leq \frac{2}{9 L}$, $\alpha \leq \min\left\{ 2 \stepsize L, \frac{ \stepsize \mu_c^2 \mu_H }{2} \right\}$, we have for all $T \geq 0$
		\begin{eqnarray*}
			\Lambda_T &\leq&  (1- \alpha)^{T} \Lambda_0  + \frac{ 4 L \stepsize^2 \sigma^2 }{ \alpha }  . 
		\end{eqnarray*}
		Fix $\varepsilon > 0$, and set the step-size in \eqref{eq:PSGD} as 
		$\stepsize = \min\left\{ \frac{2}{9 L} , \frac{\mu_c^2  \mu_H \varepsilon}{10 L \sigma^2 }  \right\}. $
		Then $\Lambda_T \leq \varepsilon$ after
		$	T = \wt \cO\rb{ \frac{ L }{\mu_c^2 \mu_H }  + \frac{ L \sigma^2 }{\mu_c^4  \mu_H^2 } \frac{1}{\varepsilon}  } $ iterations.
	\end{theorem}

\end{tcolorbox}

 \begin{proof}
The proof follows from \Cref{le:descent_like_Proj_SGD} using the same steps as in the proof of \Cref{thm:subgrad_SC}. 
 \end{proof}

Similarly to \cref{cor:subgrad_SC}, we can translate convergence in $\Lambda_T$ to the last iterate  convergence in terms of distance to the optimal solution. 
 \begin{corollary}\label{cor:Proj_SGD_SC}
	Let the assumptions of \Cref{thm:Proj_SGD_SC} hold and $x^T$ be the output of the method \eqref{eq:PSGD} after $T$ iterations (given by \Cref{thm:Proj_SGD_SC}). Then $ \Exp{ \sqnorm{x^T - x^*} } \leq \rb{ \frac{4}{\mu_H \mu_c^2} + \frac{1}{L} } \varepsilon $. 
	\end{corollary}

\Cref{thm:Proj_SGD_SC} and \Cref{cor:Proj_SGD_SC} imply that if $\mu_H >0$, Projected SGD converges linearly in deterministic setting (when $\sigma = 0$) and achieves $\wt \cO(\varepsilon^{-1})$ sample complexity in the stochastic setting. This means that compared to the special case of strongly convex optimization, the above rates have the same dependence on $\varepsilon$ (up to a logarithmic factor) \cite{stich2019unified,lan2020first}.

	\section{Projected SGD with Momentum}\label{sec:PSGDM}
    We observe that the previous section only guarantees convergence of $\Lambda_t$, however, this might not directly imply the convergence on the original function $F(\cdot)$ since $\Phi_{1/\rho}(x) \leq F(x)$ for any $x\in \cX$. It is known that in convex optimization, momentum is often helpful to establish the last iterate convergence, see e.g., \cite{Li_On_Last_Iterate_Mom_2022,Sebbouh_AS_Conv_SHB_2021}. Motivated by this, we consider Projected SGD with Polyak's (heavy-ball) momentum \cite{polyak64_HB} in the smooth setting. We show that with extra momentum step, we can establish last-iterate convergence to an $\varepsilon$-optimal solution. The Projected SGD with momentum admits the following updates: 

	\begin{tcolorbox}[colback=white,colframe=gray!75!black]
	\begin{align}
		x^{t+1} = \Pi_{\mathcal{X}}(x^t-\stepsize \, g^t) , \qquad  \label{eq:PSGDM_x_upd}
		g^{t+1} = (1-\momentum) \, g^{t} + \momentum \, \nabla f(x^{t+1}, \xi^{t+1}) . 
\end{align}
\end{tcolorbox}

Our analysis in this section uses the same properties presented in \Cref{subsec:key_ineq}, but the Lyapunov function used here is completely different from $\Lambda_t$ used in \Cref{sec:subgradient_method,sec:PSGD}. 
Let $x^* \in \cX^*$, for any $x^t \in \cX$, we define the Lyapunov function
\begin{equation}\label{eq:Lyapunov_HB}
	\Lambda_t^{HB} := \Exp{ F(x^t) - F(x^*) + \frac{\stepsize}{\momentum} \sqnorm{g^t - \nabla F(x^t)} } .
\end{equation}

The following lemma controls the error between the momentum gradient estimator $g^t$ and the true gradient $\nabla F(x^t)$. Similar recursive error control was previously used in general non-convex optimization, e.g., in \cite{cutkosky2019momentum,cutkosky_nsgdm_2020,fatkhullin2023momentum}.
 
	\begin{lemma}[Lemma 2 in \cite{fatkhullin2023momentum}]\label{le:key_HB_recursion}
		Let $\momentum \in (0,1]$ and $g^t$ be updated via \eqref{eq:PSGDM_x_upd}. Then
		\begin{eqnarray}\notag 
			\Exp{ \sqnorm{g^{t+1} - \nabla F(x^{t+1})} }  \leq (1 - \momentum) \Exp{ \sqnorm{  g^{t} - \nabla F(x^{t}) } }  + \fr{3 L^2}{\momentum} \Exp{ \sqnorm{  x^{t} - x^{t+1}  } } + \momentum^2 \sigma^2 .
		\end{eqnarray}
	\end{lemma}

The following result is the key to derive global convergence guarantee for Projected SGD with momentum under hidden convexity. 

\begin{theorem}\label{le:descent_HB_HC}
	Suppose that \ref{C1}, \ref{C2}, \ref{A1'}, \ref{A2'} hold with $\mu_H\geq0$, and the step-size in \eqref{eq:PSGDM_x_upd} satisfies $\stepsize \leq 1 / L$. For any $\alpha \in [0, 1] $, it holds that
	\begin{eqnarray}\label{eq:descent_HC}
		F(x^{t+1})  &\leq&  (1- \alpha) F(x^t) +  \alpha F(x^*)  + \rb{  \fr{ \alpha^{2} }{ \mu_c^2 \stepsize } -  \fr{(1-\alpha) \alpha  \mu_H}{2} } \sqnorm{ c(x^t) - c(x^*) },  \notag \\
		&& \qquad - \rb{ \fr{1}{2 \stepsize } - \fr{L}{2} } \sqnorm{x^{t+1}-x^t } + \frac{\stepsize}{2} \sqnorm{g^t - \nabla F(x^t) } .
	\end{eqnarray}
\end{theorem}

\begin{proof}
	By the update rule of $x^{t+1}$ and following the standard descent inequality (cf. \cref{le:bregman_properties_optimality}), we have for any $z \in \cX$ that
	\begin{eqnarray}\label{eq:projection_upd}
		\langle g^t , x^{t+1} - z \rangle  + \frac{1}{2\stepsize} \sqnorm{ x^{t+1} - x^t} \leq \frac{1}{2\stepsize} \sqnorm{ z - x^t } -\frac{1}{2\stepsize} \sqnorm{ z - x^{t+1} }  .
	\end{eqnarray}
	
	By the smoothness of $F(\cdot)$, we derive
	\begin{eqnarray}
		F(x^{t+1}) &\leq& F(x^t) + \langle \nabla F(x^t), x^{t+1}- x^t \rangle + \fr{L}{2} \sqnorm{x^{t+1}-x^t } \notag \\
		&=& F(x^t) + \langle g^t, x^{t+1}- x^t \rangle + \fr{1}{2 \stepsize } \sqnorm{x^{t+1}-x^t } \notag \\
		&& \qquad  + \langle \nabla F(x^t) - g^t , x^{t+1}- x^t \rangle - \rb{ \fr{1}{2 \stepsize } - \fr{L}{2} } \sqnorm{x^{t+1}-x^t } \notag \\
		&\overset{(i)}{\leq}&  F(x^t) + \langle g^t, z - x^t \rangle + \fr{1}{2 \stepsize } \sqnorm{z - x^t } - \fr{1}{2 \stepsize } \sqnorm{z - x^{t+1} }  \notag \\
		&& \qquad  + \langle \nabla F(x^t) - g^t , x^{t+1}- x^t \rangle - \rb{ \fr{1}{2 \stepsize } - \fr{L}{2} } \sqnorm{x^{t+1}-x^t }   \notag \\
		& { = } & F(x^t) + \langle \nabla F(x^t), z - x^t \rangle + \fr{1}{2 \stepsize } \sqnorm{ z - x^t } - \fr{1}{2 \stepsize } \sqnorm{ z - x^{t+1} }  \notag \\  
		&& \qquad + \langle \nabla F(x^t) - g^t , x^{t+1}- z \rangle - \rb{ \fr{1}{2 \stepsize } - \fr{L}{2} } \sqnorm{x^{t+1}-x^t } \notag  \\
		& \overset{(ii)}{ \leq } & F(x^t) + \langle \nabla F(x^t), z - x^t \rangle + \fr{1}{2 \stepsize } \sqnorm{z - x^t }   + \frac{\stepsize}{2} \sqnorm{g^t - \nabla F(x^t) }  \notag \\
		&& \qquad - \rb{ \fr{1}{2 \stepsize } - \fr{L}{2} } \sqnorm{x^{t+1}-x^t } \notag  \\
		&\overset{(iii)}{\leq}& F(z) + \fr{L}{2}\sqnorm{ z - x^t } + \fr{1}{2 \stepsize } \sqnorm{z - x^t }  + \frac{\stepsize}{2} \sqnorm{g^t - \nabla F(x^t) }  \notag \\
		&& \qquad - \rb{ \fr{1}{2 \stepsize } - \fr{L}{2} } \sqnorm{x^{t+1}-x^t } \notag  ,
	\end{eqnarray}
	where $(i)$ follows from \eqref{eq:projection_upd}, $(ii)$ holds by Young's inequality, i.e., $\langle a, b\rangle \leq \frac{\stepsize}{2}\sqnorm{a} + \frac{1}{2 \stepsize} \sqnorm{b} $ with $a = \nabla F(x^t) - g^t$, $b = x^{t+1} - z $, $(iii)$ holds by the smoothness of $F(\cdot)$, i.e., $-\frac{L}{2}\sqnorm{z - x^t } \leq F(x^t) - F(z) - \langle \nabla F(x^t) , z - x^t \rangle $. 
	
	We are now ready to utilize the properties of hidden convex functions to bound $F(z)$ and $\sqnorm{z - x^t}$ for some specific choice of $z \in \cX$. We select $z := x_{\alpha}^t = c^{-1}((1-{\alpha})c(x^t)+ {\alpha} c(x^*)) \in \cX$, for some $\alpha \in [0,1]$, and $x^* \in \cX^* $. By \Cref{prop:HC}, we have for $\mu_H \geq 0$, 
	$$
	F(z) \leq (1- \alpha) F(x^t) +  \alpha F(x^*) -    \fr{(1-\alpha) \alpha  \mu_H}{2} \sqnorm{ c(x^t) - c(x^*) } ,
	$$
 and
	$$
	\sqnorm{z - x^t } \leq \fr{ \alpha^{2} }{ \mu_c^2  } \sqnorm{ c(x^t) - c(x^*) }.
	$$
	Combining the three inequalities above  and utilizing the assumption $\stepsize \leq 1/L$, we complete the proof.
\end{proof}

 \subsection{Hidden Convex Setting} 
	Combining \cref{le:descent_HB_HC} with \cref{le:key_HB_recursion}, we obtain the following theorem. 
 
\begin{tcolorbox}[colback=gray!5!white,colframe=gray!75!black]
	
\begin{theorem}\label{thm:SGDM}
	Let \ref{C1}, \ref{C2}, \ref{A1'}, \ref{A2'} hold with $\mu_H = 0$, and the set $\cU$ be bounded by a diameter $D_{\cU}$. Then for any $\stepsize \leq \frac{\momentum}{4 L}$, $\momentum \in (0, 1]$, and $\alpha \leq  \frac{\momentum}{2} $, we have for any $T \geq 0$
	$$
	\Lambda_{T}^{\text{HB}} \leq (1 -  \alpha )^T \Lambda_0^{\text{HB}}  + \fr{ \alpha D_{\cU}^2}{ \mu_c^2   \stepsize }   +  \fr{\momentum \stepsize  \sigma^2 }{  \alpha  } ,
	$$
	where $\Lambda_t^{\text{HB}}$ is given by \eqref{eq:Lyapunov_HB}.
	Fix $\varepsilon > 0$, and set the parameters of algorithm \eqref{eq:PSGDM_x_upd} as 
	$$
	\stepsize = \frac{\momentum}{4 L}, \quad \beta = \min\left\{1, \frac{\mu_c^2  }{9 D_{\cU}^2 \sigma^2} \varepsilon^2 \right\}  .
	$$
	Then the scheme \eqref{eq:PSGDM_x_upd} returns a point $x^T\in\cX$ with $\Exp{F(x^{T}) - F(x^*) } \leq \varepsilon$ when
	\begin{eqnarray} \notag
	  	T = \wt \cO\rb{  \frac{L D_{\cU}^2 }{\mu_c^2  } \frac{1}{\varepsilon} + \frac{ L D_{\cU}^4 \sigma^2 }{\mu_c^4  } \frac{1}{\varepsilon^3}  } . 
	\end{eqnarray}
\end{theorem}

\end{tcolorbox}


\begin{proof} 
	By \Cref{le:descent_HB_HC}, subtracting $F(x^*)$ from both sides of  \eqref{eq:descent_HC}, setting $\mu_H=0$, and taking the expectation, we have for any $\stepsize \leq 1/L$ that
	\begin{eqnarray*}
		\Exp{F(x^{t+1}) - F(x^*)}  &\leq&  (1-  \alpha) \Exp{ F(x^t)  - F(x^*) }  +  \fr{ \alpha^{2} }{ \mu_c^2 \stepsize } \Exp{ \sqnorm{ c(x^t) - c(x^*) } } \notag \\
		&& \qquad - \rb{ \fr{1}{2 \stepsize } - \fr{L}{2} } \Exp{\sqnorm{x^{t+1}-x^t }} + \frac{\stepsize}{2} \Exp{\sqnorm{g^t - \nabla F(x^t) }} \\
		&\leq&  (1-  \alpha) \Exp{ F(x^t)  - F(x^*) }  +  \fr{ \alpha^{2} D_{\cU}^2  }{ \mu_c^2 \stepsize }  \notag \\
		&& \qquad - \rb{ \fr{1}{2 \stepsize } - \fr{L}{2} } \Exp{\sqnorm{x^{t+1}-x^t }} + \frac{\stepsize}{2} \Exp{\sqnorm{g^t - \nabla F(x^t) }} ,
	\end{eqnarray*}
	where the second inequality uses boundedness of  $\cU$.
	

Summing up the inequality  above with a $\frac{\stepsize}{\momentum}$ multiple of the result of \Cref{le:key_HB_recursion}, we recognize the Lyapunov function $\Lambda_t^{\text{HB}} $ defined in \eqref{eq:Lyapunov_HB}, and derive
	\begin{eqnarray*}
		\Lambda_{t+1}^{\text{HB}} &\leq& \Lambda_t^{\text{HB}} -  \alpha \Exp{ F(x^t)  - F(x^*) }   - \frac{\stepsize}{2} \Exp{\sqnorm{g^t - \nabla F(x^t)}} \\
		&&\qquad  + \fr{ \alpha^{2} D_{\cU}^2 }{ \mu_c^2 \stepsize }   - \rb{ \fr{1}{2 \stepsize } - \fr{L}{2} - \frac{3 L^2 \stepsize }{\momentum^2} } \Exp{\sqnorm{x^{t+1}-x^t }} + \momentum \stepsize \sigma^2 \\
		&\leq& (1 -  \alpha) \Lambda_t^{\text{HB}} + \fr{ \alpha^{2} D_{\cU}^2 }{ \mu_c^2 \stepsize }  + \momentum \stepsize \sigma^2 ,
	\end{eqnarray*}
	where the last step holds for $ \alpha \leq \momentum / 2$ and $\stepsize \leq \frac{\momentum}{4 L}$. Unrolling the recursion from $t = 0$ to $t = T-1$ and choosing $\stepsize = \fr{\momentum}{ 4 L }$, we obtain
	\begin{eqnarray}
		\Lambda_{T}^{\text{HB}} &\leq& (1 -  \alpha )^T \Lambda_0^{\text{HB}}  + \fr{ \alpha D_{\cU}^2}{ \mu_c^2   \stepsize }   +  \fr{\momentum \stepsize  \sigma^2 }{  \alpha  } \notag \\
		&\leq& (1 -  \alpha )^T \Lambda_0^{\text{HB}}  + \fr{ 4 L  D_{\cU}^2  }{ \mu_c^2  } \frac{\alpha}{\momentum}  +  \fr{ \sigma^2 }{ 4 L   } \frac{\momentum^2}{\alpha} \leq \varepsilon  \notag ,
	\end{eqnarray}
	where the last inequality holds by setting $\alpha =  \min\left\{ \frac{\momentum}{2}, \frac{3 \mu_c^2   }{2 L D_{\cU}^2} \momentum \varepsilon , \frac{\sigma \mu_c }{4 L D_{\cU}} \beta^{\frac{3}{2}} \right\}$ , $\beta = \min\left\{1, \frac{\mu_c^2  }{9 D_{\cU}^2 \sigma^2} \varepsilon^2 \right\} $, and the number of iterations as 
	$$
	T = \frac{1}{ \alpha} \log\rb{\frac{3 \Lambda_0^{\text{HB}}}{\varepsilon} } = \wt \cO\rb{ \frac{L D_{\cU}^2 }{\mu_c^2  } \frac{1}{\varepsilon}  + \frac{ L D_{\cU}^4 \sigma^2 }{\mu_c^4 } \frac{1}{\varepsilon^3}  } .
	$$
\end{proof}

\subsection{Hidden Strongly Convex Setting}

We conclude the section with the improved result for Projected SGD with momentum under hidden strong convexity.

\begin{tcolorbox}[colback=gray!5!white,colframe=gray!75!black]

\begin{theorem}\label{thm:SGDM_SC}
	Let \ref{C1}, \ref{C2}, \ref{A1'}, \ref{A2'} hold with $\mu_H > 0$. Then for any $\stepsize \leq \frac{\momentum}{4 L}$, $\momentum \in (0, 1]$, and $\alpha \leq  \min\left\{ \frac{\momentum}{2} , \frac{\mu_c^2 \mu_H \stepsize}{4} \right\} $, we have for any $T \geq 0$
	$$
	\Lambda_{T}^{\text{HB}} \leq (1 -  \alpha )^T \Lambda_0^{\text{HB}}  +  \fr{\momentum \stepsize  \sigma^2 }{  \alpha  } ,
	$$
	where $\Lambda_t^{\text{HB}}$ is given by \eqref{eq:Lyapunov_HB}.
	Fix $\varepsilon > 0$, and set the parameters of algorithm \eqref{eq:PSGDM_x_upd} as 
	$$
	\stepsize = \frac{\momentum}{4 L}, \quad \beta = \min\left\{1, \frac{\mu_c^2 \mu_H }{8 \sigma^2} \varepsilon \right\}  .
	$$
	Then the scheme \eqref{eq:PSGDM_x_upd} returns a point $x^T\in\cX$ with $\Exp{F(x^{T}) - F(x^*) } \leq \varepsilon$ after
	\begin{eqnarray} \notag
		T = \wt \cO\rb{  \frac{L }{\mu_c^2 \mu_H  }  + \frac{ L \sigma^2 }{\mu_c^4 \mu_H^2 } \frac{1}{\varepsilon}  } . 
	\end{eqnarray}
\end{theorem}

\end{tcolorbox}

\begin{proof}
    Applying \Cref{le:descent_HB_HC} with $\mu_H > 0$, and setting $\alpha$ small enough allows us to cancel the term invloving $\sqnorm{c(x^t) - c(x^*)}$.
    The rest of the proof is similar to the one of \Cref{thm:Proj_SGD_SC}. 
\end{proof}

We remark that both \Cref{thm:SGDM,thm:SGDM_SC} provide last iterate global convergence for Projected SGD with momentum without the need of using large mini-batch. Additionally, the gradient estimate $g^t$ is guaranteed to converge to the true gradient $\nabla F(x^*)$ at the optimum $x^*\in \cX^*$, which might be non-zero when minimizing over a compact set $\cX$. In the hidden strongly convex case, similarly to \Cref{cor:subgrad_SC,cor:Proj_SGD_SC}, the result of \Cref{thm:SGDM_SC} can be translated to the point convergence to the optimal solution.

	\section{Conclusions}
	\label{sec:conclusions}
	In this work, we study stochastic optimization under hidden convexity and develop sample complexity results for batch-free stochastic (sub-)\,gradient methods with projection. 
 
Several questions remain open. 1) We know that in case $\mu_H > 0$, the derived sample complexity is worst-case optimal (up to the logarithmic factor) in terms of dependence on $\varepsilon$  since it matches the optimal rate known for strongly convex $F(\cdot)$, and therefore, the complexity bounds are unimprovable for SM and P-SGD. However, for merely convex $H(\cdot)$, i.e., $\mu_H = 0$, it is unclear if our $\wt \cO(\varepsilon^{-3})$ sample complexity is tight for SM and P-SGD. 2). The benefits of momentum variants of P-SGD can be further explored, e.g., to understand if Nesterov's acceleration is possible under hidden convexity. 3) When $\mu_H = 0$, our iteration and sample complexity results depend on the diameter of the reformulated problem. It would be interesting to explore if $D_{\cU}$ can be replaced with the distance to the solution, i.e., $\norm{c(x^0) - c(x^*)}$. 

There are also many other directions to explore in the future. 1) SM and P-SGD are the simplest and generic methods for solving \eqref{problem:original}. It is important to explore more advanced specialized algorithms for applications, which may potentially speed up the convergence. For instance, given a stochastic information about the map $c(x) = \Exp{c(x, \xi)}$, one can utilize the samples $c(x, \xi)$ or $\nabla c(x, \xi)$ in the algorithm. Despite some recent progress \cite{chen2022efficient}, the rigorous validation of such methods remains an open problem with a general convex constrained $\cX$. 2) The development of stochastic gradient methods for solving hidden convex problems with non-convex (e.g., hidden convex) constraints is an interesting research direction \cite{xu2021crpo,ding2022Nat_PDPG4constr_MDP,zhao2023global_safeLQR,li2021faster_CMDP}. 3) Extension of our results to hidden convex saddle point problems and games \cite{pattathil2023symmetric_optim_NPG} also merits further exploration.

	\appendix

	\section{Technical Lemma}\label{sec:technical_lemma}
	We report the following technical lemma from \cite{davis2019stochastic,fatkhullin2023momentum} and include their slightly modified proofs for completeness. 
 
	\begin{lemma}[Lemma 3.2 in \cite{davis2019stochastic}]\label{le:prox_grad_prox_relation}
		Let $\rho > \ell$, and for any $x^t \in \cX$, define $\hat{x}^t := \operatorname{prox}_{ \Phi/\rho }(x^t)$, where $\Phi := F + \delta_{\cX} $. Then
		$
		\hat{x}^t= \Pi_{\cX} \left(\stepsize \rho x^t-\stepsize \hat g^t + (1-\stepsize \rho) \hat{x}^t\right) ,
		$ where $\hat g^t \in \partial F(\hat x^t)$.
	\end{lemma} 
	\begin{proof}
		By definition of $\hat{x}^t$ and $\Phi(\cdot) $, we have
		$$
		0 \in  \partial \left( F + \frac{\rho}{2} \sqnorm{ \cdot - x^t } + \delta_{\cX} \right)(\hat x^t) =  \hat g^t + \rho \left(\hat{x}^t-x^t\right)  + \partial \delta_{\cX} \left(\hat{x}^t\right) ,
		$$
		where the last equality holds, since $F(\cdot) + \frac{\rho}{2}\sqnorm{\cdot - x^t } $, and $\delta_{\cX}(\cdot)$ are both convex (due to the conic combination rule). 
		Multiplying both sides by $\stepsize > 0$ and rearranging, we get
		$
		z^t := \stepsize \rho x^t-\stepsize  \hat g^t  + (1-\stepsize \rho) \hat{x}^t \in \hat{x}^t+\stepsize \partial \delta_{\cX} \left(\hat{x}^t\right) .
		$
		Therefore, by the optimality condition for the proximal sub-problem, we have
		$ \hat{x}^t=\operatorname{prox}_{\stepsize \delta_{\cX}}\left(z^t\right)  = \Pi_{\cX}(z^t)$.
	\end{proof}

\begin{customproof}{\Cref{le:prox_grad_prox_bound_nonsmooth}}
	Lemma~\ref{le:prox_grad_prox_relation} states that for any $\hat g^t\in \partial F(\hat x^t)$ and $z^t = \stepsize \rho x^t-\stepsize \hat g^t + (1-\stepsize \rho) \hat{x}^t$, we have $\hat{x}^t = \Pi_{\cX}(z^t)$. Thus, using the update rule of $x^{t+1}$ and non-expansiveness of the projection, we derive 
	$$
		\begin{aligned}
			&\mathbb{E}\left[\sqnorm{ x^{t+1}-\hat{x}^t } \mid x^t\right] =\mathbb{E}\left[\sqnorm{ \Pi_{\cX}\left(x^t-\stepsize g(x^t, \xi^t) \right) - \Pi_{\cX} \left(z^t\right) } \mid x^t\right] \\
			& \leq \mathbb{E}\left[ \sqnorm{ x^t-\stepsize g(x^t, \xi^t) - \left(\stepsize \rho x^t-\stepsize \hat g^t +  (1-\stepsize \rho) \hat{x}^t\right) } \mid x^t\right] \\
			& =\mathbb{E}\left[ \sqnorm{ (1-\stepsize \rho)\left(x^t-\hat{x}^t\right)-\stepsize\left( g(x^t, \xi^t) - \hat g^t  \right) } \mid x^t\right] \\
			& \overset{(i)}{ = }  (1-\stepsize \rho)^2 \sqnorm{ x^t-\hat{x}^t } -  2 (1-\stepsize \rho) \stepsize  \langle g^t - \hat g^t ,  x^t-\hat{x}^t  \rangle   + \stepsize^2 \Exp{ \sqnorm{ g(x^t, \xi^t) - \hat g^t } \mid x^t }\\
						& \overset{(ii)}{\leq}   (1-\stepsize \rho)^2 \sqnorm{ x^t-\hat{x}^t } -  2 (1-\stepsize \rho) \stepsize  \langle g^t - \hat g^t ,  x^t-\hat{x}^t  \rangle   + 4 G_F^2 \stepsize^2 \\
			& \overset{(iii)}{\leq}   (1-\stepsize \rho)^2 \sqnorm{ x^t-\hat{x}^t } +  2 (1-\stepsize \rho) \stepsize  \ell \sqnorm{  x^t - \hat{x}^t }   + 4 G_F^2 \stepsize^2 \\
			& = (1-\stepsize \rho) \sqnorm{ x^t-\hat{x}^t } + 4 G_F^2 \stepsize^2 ,
		\end{aligned}
	$$
	where in $(i)$ we use unbiasedness of the gradient estimator. In $(ii)$, we use Young's inequality and A.2, $(iii)$ holds by hypo-monotonicity inequality $\langle g^t - \hat g^t,  x^t - \hat x^t \rangle \geq - \ell \sqnorm{x^t - \hat x^t}$. The last equality holds by the choice of $\rho$.
\end{customproof}

\begin{lemma}[Lemma 3.4 in \cite{davis2019stochastic}]\label{le:prox_grad_prox_bound}
		Let \ref{A1'}, \ref{A2'} hold, and $\rho=4 L$, $\stepsize \leq \frac{2}{9 L }$. Then for all $t \geq 0$
		$$
		\mathbb{E}\left[ \sqnorm{ x^{t+1}-\hat{x}^t } \mid x^t\right] \leq(1-\stepsize \rho)\sqnorm{x^t-\hat{x}^t } + \sigma^2 \stepsize^2
		$$
	\end{lemma}
	\begin{proof}
		For a differentiable $F(\cdot)$ Lemma~\ref{le:prox_grad_prox_relation} implies that for $z^t = \stepsize \rho x^t-\stepsize \nabla F\left(\hat{x}^t\right)+(1-\stepsize \rho) \hat{x}^t$, we have $\hat{x}^t = \Pi_{\cX}(z^t)$. Thus, using the update rule of $x^{t+1}$ and non-expansiveness of the projection, we derive 
		$$
		\begin{aligned}
			&\mathbb{E}\left[\sqnorm{ x^{t+1}-\hat{x}^t } \mid x^t\right] =\mathbb{E}\left[\sqnorm{ \Pi_{\cX}\left(x^t-\stepsize \nabla f\left(x^t, \xi^{t}\right)\right) - \Pi_{\cX} \left(z^t\right) } \mid x^t\right] \\
			& \leq \mathbb{E}\left[ \sqnorm{ x^t-\stepsize \nabla f\left(x^t, \xi^{t} \right)-\left(\stepsize \rho x^t-\stepsize \nabla F\left(\hat{x}^t\right)+(1-\stepsize \rho) \hat{x}^t\right) } \mid x^t\right] \\
			& =\mathbb{E}\left[ \sqnorm{ (1-\stepsize \rho)\left(x^t-\hat{x}^t\right)-\stepsize\left(\nabla f\left(x^t, \xi^{t}\right)-\nabla F\left(\hat{x}^t\right)\right) } \mid x^t\right] \\
			& =\mathbb{E}\left[ \sqnorm{ (1-\stepsize \rho)\left(x^t-\hat{x}^t\right)-\stepsize\left(\nabla F\left(x^t\right)-\nabla F\left(\hat{x}^t\right)\right)-\stepsize\left(\nabla f\left(x^t, \xi^{t}\right)-\nabla F\left(x^t\right)\right) } \mid x^t\right] \\
			& \overset{(i)}{=}  \sqnorm{ (1-\stepsize \rho)\left(x^t-\hat{x}^t\right)-\stepsize\left(\nabla F\left(x^t\right)-\nabla F\left(\hat{x}^t\right)\right) } 
			+\stepsize^2 \mathbb{E} \left[\sqnorm{ \nabla f\left(x^t,\xi^{t}\right)-\nabla F\left(x^t\right) } \mid x^t\right] \\
			& \overset{(ii)}{\leq} \sqnorm{ (1-\stepsize \rho)\left(x^t-\hat{x}^t\right)-\stepsize\left(\nabla F\left(x^t\right)-\nabla F\left(\hat{x}^t\right)\right) } + \stepsize^2 \sigma^2 \\
			& =(1-\stepsize \rho)^2 \sqnorm{ x^t-\hat{x}^t } - 2 (1-\stepsize \rho) \stepsize\left(x^t-\hat{x}^t, \nabla F\left(x^t\right)-\nabla F\left(\hat{x}^t\right)\right\rangle+\stepsize^2 \sigma^2 \\
			& \quad+\stepsize^2 \sqnorm{ \nabla F\left(x^t\right)-\nabla F\left(\hat{x}^t\right) }  \\
			& \overset{(iii)}{\leq} (1-\stepsize \rho)^2 \sqnorm{ x^t-\hat{x}^t } + 2(1-\stepsize \rho) \stepsize L \sqnorm{ x^t-\hat{x}^t } + \stepsize^2 L^2 \sqnorm{ x^t-\hat{x}^t } +\stepsize^2 \sigma^2 \\
			& =(1-\stepsize \rho)\left(1-\stepsize \rho+2 \stepsize L+\frac{\stepsize^2 L^2}{1-\stepsize \rho}\right) \sqnorm{ x^t-\hat{x}^t } + \stepsize^2 \sigma^2 \\
			& \leq(1-\stepsize \rho) \sqnorm{ x^t-\hat{x}^t } + \stepsize^2 \sigma^2,
		\end{aligned}
		$$
		where in $(i)$ and $(ii)$ use unbiasedness of the gradient estimator and bounded variance. In $(iii)$, we use Cauchy–Schwarz inequality and smoothness of $F(\cdot)$, i.e., $\norm{\nabla F\left(\hat{x}^t\right)-\nabla F\left(x^t\right) } \leq L \norm{\hat{x}^t-x^t } $. The last inequality holds by the choice of $\rho, \stepsize$ and $2 \stepsize L \leq \frac{\stepsize \rho}{2}$, and $\frac{\stepsize^2 L}{1-\stepsize \rho} \leq \frac{\stepsize \rho}{2}$.
	\end{proof}

	\begin{customproof}{\Cref{le:key_HB_recursion}}
		Using the update rule of $g^{t+1}$ and the unbiasedness of stochastic gradients, we have
  $$
		\begin{aligned}
			\mathbb{E}&\left[ \sqnorm{g^{t+1} - \nabla F(x^{t+1})} \right] 
			 =  \Exp{ \sqnorm{(1 - \momentum) g^{t} + \momentum \nabla f(x^{t+1}, \xi^{t+1}) - \nabla F(x^{t+1})}  } \notag \\
			& =  (1-\momentum)^2 \Exp{ \sqnorm{  g^{t} - \nabla F(x^{t+1}) } }  +  \momentum^2 \Exp{\sqnorm{\nabla f(x^{t+1}, \xi^{t+1}) - \nabla F(x^{t+1}) }} \notag \\
			& \leq   (1 - \momentum)^2 \rb{ 1 + \nfr{\momentum}{2}} \Exp{ \sqnorm{  g^{t} - \nabla F(x^{t}) } }  + \rb{1+\nfr{2}{\momentum} } \Exp{ \sqnorm{  \nabla F(x^{t}) - \nabla F(x^{t+1}) }  } + \momentum^2 \sigma^2  \notag \\
			& \leq   (1 - \momentum) \Exp{ \sqnorm{  g^{t} - \nabla F(x^{t}) } }  + \fr{3 L^2}{\momentum} \Exp{ \sqnorm{  x^{t} - x^{t+1}  } } 
   + \momentum^2 \sigma^2  \notag ,
		\end{aligned}
  $$
		where the first inequality uses Young's inequality and the bound of the variance of stochastic gradients, and the last step uses the Lipschitz continuity of the gradient and the fact that $(1 - \momentum) \rb{ 1 + \nfr{\momentum}{2}} \leq 1 $ for all $\momentum \in (0,1]$.
	\end{customproof}

The following technical lemma is fairly standard, e.g., see \cite{Guler_PPM}.
	
	\begin{lemma}\label{le:bregman_properties}
			Let $\phi(\cdot)$ be convex and for any $\stepsize>0$, $x\in \cX$, define $$x^{+} := \argmin_{y\in \cX} \left\{ \phi(y) + \fr{1}{2 \stepsize} \sqnorm{y - x} \right\}. $$ 
        Then
			$$
			\phi(y) + \fr{1}{2 \stepsize} \sqnorm{y - x} \geq \phi(x^{+}) + \fr{1}{2 \stepsize} \sqnorm{x^{+} - x }  + \fr{1}{2\stepsize} \sqnorm{ y  - x^{+} }  \quad \text{for all } y \in \cX. 
			$$\label{le:bregman_properties_optimality}
	\end{lemma}

 \section{Further Improvements with Mini-batching}\label{sec:minibatching_appendix}
Using only one (post-processing) step of mini-batch P-SGD to the output of one batch P-SGD is sufficient to translate convergence from $\Phi_{1/4L}(\cdot)$ to $F(\cdot)$.
    \begin{corollary}\label{cor:Proj_SGD_C}
	Let the assumptions of \Cref{thm:Proj_SGD_C} hold and $G_F > 0 $ be the Lipschitz constant of $F(\cdot)$ over $\cX$. Set $x^{T+1} = \Pi_{\cX}\rb{x^T - \frac{1}{3 L} \frac{1}{B_0} \sum_{i=1}^{B_0} \nabla F(x^T, \xi_i^T) } $, where $B_0 \geq \min\{1, \rb{\frac{G_F \sigma}{3 L \varepsilon}}^2 \}$, $x^T$ is the output of \eqref{eq:PSGD} applied with batch-size $B=1$ after $T$ iterations (given by \Cref{thm:Proj_SGD_C}). Then $\Exp{F(x^{T+1}) - F(x^*)} \leq 2 \, \varepsilon$.
	\end{corollary}
	\begin{proof}
        Define $x_+^T := \Pi_{\cX}(x^T - \frac{1}{\rho - L} \nabla F(x^T))$, $\rho = 4 L$.
		Notice that $F(x_+^T) =  \Phi(x_+^T) \leq \Phi_{1/\rho}(x^T)$, where the inequality follows by \cite[Poposition 2.5-(\textit{i})]{Stella_Themelis_Patrinos_2017} with $\gamma := (\rho - L)^{-1}$. Therefore, \Cref{thm:Proj_SGD_C} implies that
		$$
		\Exp{F(x_+^{T}) - F(x^*) } \leq \Exp{\Phi_{1/\rho}(x^T)  - F(x^*) }  \leq \varepsilon ,
		$$
  On the other hand, the post-processing step guarantees
  $$
    \begin{aligned}
		 \mathbb E & \left[ F(x^{T+1}) - F(x_+^T) \right] \leq G_F \, \Exp{\norm{ x^{T+1} - x_+^T } } \\
            &\leq  G_F \Exp{\norm{ \Pi_{\cX}\rb{x^T - \frac{1}{3 L} \frac{1}{B_0} \sum_{i=1}^{B_0} \nabla F(x^T, \xi_i^T) } - \Pi_{\cX}\rb{x^T - \frac{1}{3 L}  \nabla F(x^T) } }} \\
            &\leq  \frac{G_F}{3 L}\Exp{ \norm{  \frac{1}{B_0} \sum_{i=1}^{B_0} \nabla F(x^T, \xi_i^T) -   \nabla F(x^T) } } \leq \frac{G_F \sigma }{ 3 L \sqrt{B_0}} \leq \varepsilon.
	\end{aligned}
  $$
  Combining the above two inequalities, the result follows. 
	\end{proof}

 The following corollary shows that if we apply mini-batching at each iteration with sufficiently large batch-size, then the number of iterations required for convergence is reduced to $\wt \cO(\varepsilon^{-1})$. 

\begin{corollary}
    \label{cor:Proj_SGD_C_minibatching}
	Let the assumptions of \Cref{thm:Proj_SGD_C} hold and $G_F > 0 $ be the Lipschitz constant of $F(\cdot)$ over $\cX$. Suppose P-SGD with batch-size $B$ is applied, i.e., $\{x^t\}_{t\geq0}$ is generated by  $x^{t+1} = \Pi_{\cX}\rb{x^t - \stepsize \frac{1}{B} \sum_{i=1}^{B} \nabla F(x^t, \xi_i^t) } $  with $\eta = \frac{2}{9 L}$, $B \geq \min\{1, \frac{D_{\cU}^2 \sigma^2}{\mu_c^2 \varepsilon^2} \}$. Define $x^{T+1} = \Pi_{\cX}\rb{x^T - \frac{1}{3 L } \frac{1}{B_0} \sum_{i=1}^{B_0} \nabla F(x^T, \xi_i^T) } $, where $\rho = 4 L$, and $B_0 \geq \min\{1, \rb{\frac{G_F \sigma}{3 L \varepsilon}}^2 \} $. Then $\Exp{F(x^{T+1}) - F(x^*)} \leq 2 \, \varepsilon$ after
        $
 T = \wt \cO \rb{ \frac{L D_{\cU}^2}{\mu_c^2  \varepsilon} } .	
  $
\end{corollary}

\begin{proof}
    The proof follows from the previous corollary by replacing $\sigma^2$ with $\sigma^2 / B$.
\end{proof}

However, we highlight that the results of \cref{thm:Proj_SGD_C,cor:Proj_SGD_C} do not require using large batches of samples at every iteration. 





    \section{Historical Remarks}
     
    The sub-gradient method, its special case, Projected SGD, and their numerous variants have a long history of development since the first works on stochastic approximation appeared in 1950s \cite{robbins1951stochastic,Kiefer_Wolfowitz_1952,blum1954multidimensional,chung1954stochastic}. 
	
	\textit{Convex optimization.} The case of convex $F(\cdot)$ is particularly well documented \cite{moulines2011non,agarwal2009information,fontaine2021convergence}. Researchers have studied how to deal with convex constraints, proximal operators, general Bregman divergences \cite{nemirovski1983problem,beck2003mirror}, and leveraging averaging and momentum schemes \cite{polyak1992acceleration,Sebbouh_AS_Conv_SHB_2021,gadat2018_SHB,Li_On_Last_Iterate_Mom_2022}. In the convex case, the global convergence of gradient methods in the function value, i.e., find $x\in \cX$ with $\Exp{F(x) - F(x^*)}\leq \varepsilon$ for any $\varepsilon>0$, is naturally possible and the sample complexity required is $\cO\rb{\varepsilon^{-2}}$.\footnote{For P-SGD under smooth and bounded variance assumptions, A.1' and A.2' in \cref{sec:HC_class}, or for SM under Lipschitz continuity and bounded second moment of stochastic sub-gradients, i.e., A.2. }
	

    \textit{Non-convex optimization.} In the last decade, the interest in the optimization community shifted towards general non-convex problems (often smooth or weakly convex), where only convergence to a FOSP is possible in general \cite{khaled2020better,drori2020complexity,arjevani2023lower,Yang_Two_Sides_2023}, i.e., find $x\in \cX$ with $\Exp{\norm{\nabla F(x) } }\leq \epsilon$ when $F(\cdot)$ is smooth. Similar to developments in convex optimization, convergence of non-convex SGD extends to constrained/proximal setting \cite{ghadimi2013stochastic,lan2020first,bottou2018optimization}, mirror descent \cite{zhang_he_2018_msgd,davis2018stochastic_high_order_growth}, momentum \cite{liu2020improved,fatkhullin2023momentum}, variance-reduction \cite{cutkosky2019momentum,arjevani2023lower}, and biased gradient setting~\cite{hu2020biased,hu2021bias,hu2023contextual}. For the more general weakly-convex case \cite{davis2019stochastic,mai2020convergence,zhang_he_2018_msgd},  the convergence guarantees are usually with respect to a gradient norm of a smoothed objective. Some works consider non-convex functions with a specific compositional structure similar to \eqref{problem:reformulation}, e.g. the composition of a convex function with a differentiable and smooth map $c(\cdot)$, see \cite{ModifiedGaussNewton_Nesterov_2003,Lewis_Wright_2015,drusvyatskiy2019efficiency,StochasticCompositeGradient2019}. Recently a number of works focus on non-convex non-smooth optimization (beyond weak convexity) and develop convergence for suitably defined notions of FOSP \cite{cutkosky23a_optimal,zhang2020complexity_nonsmooth_nonconvex,jordan2023deterministic}. Although the above works consider non-convex problems, which find a wide range of applications, they often only provide convergence to a FOSP rather than global convergence in the function value.

	\bibliographystyle{plain}
	\bibliography{references}
\end{document}

%% file: ex_shared.tex
\usepackage{amsmath,amssymb, amsthm}
\usepackage{hyperref}

\hypersetup{
	colorlinks   = true, 
	urlcolor     = blue, 
	linkcolor    = {blue!50!black}, 
	citecolor   = {blue!50!black} 
}

\usepackage{cleveref}
\usepackage{lipsum}
\usepackage{amsfonts}
\usepackage{graphicx}
\usepackage{epstopdf}
\usepackage{algorithmic}
\ifpdf
  \DeclareGraphicsExtensions{.eps,.pdf,.png,.jpg}
\else
  \DeclareGraphicsExtensions{.eps}
\fi

\newtheorem{lemma}{Lemma}
\newtheorem{theorem}{Theorem}
\newtheorem{proposition}{Proposition}
\newtheorem{definition}{Definition}
\newtheorem{corollary}{Corollary}

\title{
	\bf Stochastic Optimization under Hidden Convexity\thanks{
		This work is supported by ETH AI Center Doctoral Fellowship, NCCR Automation, and Swiss National Science Foundation. 
	}
}

\author{Ilyas Fatkhullin\thanks{Department of Computer Science, ETH Z\"urich, Switzerland
  (\href{mailto:ilyas.fatkhullin@ai.ethz.ch}{ilyas.fatkhullin@ai.ethz.ch}, \href{mailto:niao.he@inf.ethz.ch}{niao.he@inf.ethz.ch}).}
\and Niao He\footnotemark[2]
\and Yifan Hu\thanks{College of Management of Technology, EPFL and Department of Computer Science, ETH Z\"urich, Switzerland
	(\href{mailto:yifan.hu@epfl.ch}{yifan.hu@epfl.ch}).} }

\usepackage{amsopn}

\usepackage{nicefrac}

\newcommand{\cD}{\mathcal{D}}
\newcommand{\cX}{\mathcal{X}}

\newcommand{\cU}{\mathcal{U}}

\newcommand{\Exp}[1]{{ \mathbb{E}}\left[#1\right]}
\newcommand{\Expu}[2]{{\mathbb{E}}_{#1}\left[#2\right]}

\newcommand{\ve}[2]{\langle #1 ,  #2 \rangle}
\newcommand{\norm}[1]{\left\| #1 \right\|}
\newcommand{\opnorm}[1]{\left\| #1 \right\|_{\text{op}}}
\newcommand{\pnorm}[1]{\left\| #1 \right\|}
\newcommand{\R}{\mathbb R}

\DeclareMathOperator*{\argmin}{arg\,min}

\newcommand{\sqnorm}[1]{\left\| #1 \right\|^{2}}
\newcommand{\sqpnorm}[1]{\left\| #1 \right\|^2}

\newcommand{\rb}[1]{\left(#1\right)}

\newcommand{\nfr}{\nicefrac}
\newcommand{\fr}{\frac}

\newcommand{\stepsize}{\eta}
\newcommand{\momentum}{\beta}

\newcommand{\cO}{{\mathcal O}}
\newcommand{\wt}{\widetilde}